\documentclass{article}
\usepackage{import}
\usepackage[a4paper,margin=1.0in]{geometry}
\usepackage{bm}
\usepackage{psfrag}
\usepackage[usenames,dvipsnames]{color}
\usepackage{xcolor}
\usepackage[all,cmtip,line]{xy}
\usepackage{tikz}
\usepackage[utf8]{inputenc}
\usepackage{pgfplots}
\pgfplotsset{compat=1.18}
\usepackage{amssymb}
\usepackage{caption}
\usepackage{subcaption}
\usepackage{amsmath}
\usepackage{mathabx}
\usepackage[shortlabels]{enumitem}
\usepackage{comment}
\usepackage{algorithm}
\usepackage{algpseudocode}
\usepackage{indentfirst}

\usepackage{mathtools}
\numberwithin{equation}{section}

\usepackage[T3,T1]{fontenc}
\DeclareSymbolFont{tipa}{T3}{cmr}{m}{n}
\DeclareMathAccent{\invbreve}{\mathalpha}{tipa}{16}

\usepackage{amsthm}

\usepackage{cleveref}
\crefname{definition}{definition}{definitions}
\Crefname{definition}{Definition}{Definitions}
\Crefname{problem}{Problem}{Problems}
\crefname{problem}{problem}{problems}
\Crefname{proposition}{Proposition}{Propositions}
\crefname{proposition}{proposition}{propositions}
\crefname{remark}{remark}{remarks}
\Crefname{remark}{Remark}{Remarks}
\crefname{assumption}{assumption}{assumptions}
\Crefname{assumption}{Assumption}{Assumptions}

\newtheorem{theorem}{Theorem}[section]
\newtheorem{lemma}[theorem]{Lemma}
\newtheorem{proposition}[theorem]{Proposition}

\newtheorem{definition}[theorem]{Definition}

\newtheorem{remark}[theorem]{Remark}
\newtheorem{assumption}[theorem]{Assumption}
\newtheorem{conjecture}[theorem]{Conjecture}

\definecolor{forestgreen}{rgb}{0.13, 0.55, 0.13}

\newcommand{\jp}[1]{{\color{blue}#1}}
\newcommand{\ra}[1]{{\color{magenta}#1}}

\newcommand{\be}{\begin{equation}}
	\newcommand{\ee}{\end{equation}}

\usepackage{algorithm}
\usepackage{algpseudocode}

\newcommand{\C}{\mathcal C}
\newcommand{\cD}{\mathcal D}

\newcommand{\cO}{\mathcal O}

\newcommand{\cS}{\mathcal S}

\newcommand{\nno}{\nonumber}

\newcommand{\dup}[2]{\langle #1, #2\rangle}

\newcommand{\norm}[1]{\left \lVert #1 \right \rVert}
\newcommand{\snorm}[1]{\left \lvert #1 \right \rvert}

\newcommand{\IC}{{\mathbb C}}

\newcommand{\IN}{{\mathbb N}}

\newcommand{\IR}{{\mathbb R}}

\newcommand{\IP}{{\mathbb P}}

\newcommand{\IZ}{{\mathbb Z}}

\renewcommand{\d}{\!\!\operatorname{d}\!}

\newcommand{\fpcs}[1]{\operatorname{\mathsf{f}\pc}}

\newcommand{\supp}{\operatorname{supp}}

\newcommand{\abs}[1]{\left|#1\right|}

\newcommand{\pc}{}

\newcommand{\loc}{{\mathrm{\text{loc}}}}

\newcommand{\VH}{\bm{H}}

\newcommand{\bw}{\bm{w}}

\newcommand{\bv}{\bm{v}}
\newcommand{\br}{\bm{r}}

\newcommand{\bx}{\bm{x}}
\newcommand{\by}{\bm{y}}

\newcommand{\href}{\VH^{\mathrm{ref}}}

\newcommand{\half}{\frac{1}{2}}
\newcommand{\threehalf}{\frac{3}{2}}

\usepackage[utf8]{inputenc}
\usepackage[english]{babel}
\usepackage{graphicx}
\usepackage{parskip}
\usepackage{imakeidx}
\usepackage{geometry}
\title{Integral Formulations for two-dimensional Multi-Arcs}
\author{Jose Pinto\footnote{Universidad Adolfo Ibañez, jose.pinto@uai.cl} and Ruben Aylwin\footnote{Universität Ulm, ruben.aylwin-pincheira@uni-ulm.de}} \date{January 2026}

\begin{document}

\maketitle
\abstract{We study the Laplace equation with Dirichlet and Neumann boundary conditions posed on multi-arcs, i.e., collections of open arcs meeting at junction points. 
We begin by introducing a scale of Sobolev spaces constructed using the Sobolev spaces on open arcs as main building block and extend the definition of trace operators.

We reformulate the boundary value problems using boundary integral formulations. We then establish a well-posed integral formulation for the Dirichlet problem, which can be discretized using standard numerical methods.

We further investigate the singular behavior of the solution densities at branch points through numerical experiments and observe that these singularities are comparable to the corner singularities arising in polygonal domains.

For the Neumann problem, we show that the associated hypersingular operator is not necessarily invertible on classical Sobolev spaces and provide numerical evidence that solutions may develop jump discontinuities at branch points.}


\section{Introduction}

In this work, we study the Laplace equation
\begin{equation}
\begin{split} 
  -\Delta U &= 0, \quad \text{on }\IR^2 \setminus \overline{\Gamma}, \\
  U(\bx) &= \cO(1), \quad \text{as } \Vert\bx\Vert \rightarrow \infty.
  \end{split}
  \label{eq:lap}
\end{equation}
We focus on the case where $\Gamma$ is a \emph{multi-arc}, that is, a finite collection of open arcs whose pairwise intersections are either empty or consist of a single point.  Such configurations naturally give rise to geometries with corners and branch points; see, for example, \Cref{fig:NeumanGeo}.
On $\Gamma$, we impose one of the following boundary conditions: 
 $$
 U|_\Gamma = f \ \ \text{(Dirichlet Condition),} \quad \bm{n} \cdot (\nabla U)|_\Gamma = f \ \ \text{(Neumann Condition),} 
 $$
where $f$ is a given function and $\bm{n}$ denotes a unit normal vector to $\Gamma$. These conditions define what we shall refer to as the \emph{Dirichlet problem} and the \emph{Neumann problem}, respectively.

Since the problem in \cref{eq:lap} is posed on an unbounded domain, $\IR^2 \setminus \overline{\Gamma}$, with bounded boundary $\Gamma$, our analysis is carried out within the framework of Boundary Integral Equations (BIEs).

The theory of BIEs typically involves two main steps:
\begin{enumerate}
  \item Reformulating the boundary value problem as an integral equation and establishing its well-posedness.
  \item Developing accurate and efficient numerical approximations of the resulting integral equation.
\end{enumerate}
Regarding the first point, classical treatments (see, for instance, \cite{mclean2000strongly,Sauter:2011,STEI07}) mainly address partial differential equations posed on bounded Lipschitz domains, or their exterior. In such cases, the boundary is a closed curve (or closed surface in higher dimensions). Extensions to problems with boundaries given by open arcs are also available, typically by embedding the arc into a closed boundary; see, for example, \cite{stephan1984augmented} and \cite[Sec.~3.5.3]{Sauter:2011}.

Concerning the numerical approximation of boundary integral equations, many aspects must be taken into account. In this work, we focus specifically on the singular behavior of solutions to the underlying integral equations. Singularities at corners are by now well understood; see, for example, \cite{grisvard2011elliptic,costabel1985boundary,gwinner2018advanced}. Singularities arising on open arcs can be interpreted as a limiting case of corner singularities with opening angle $2\pi$, with refined analyses available in \cite{COSTABEL,costabel2003}.

For multi-arcs, however, additional difficulties arise. BIEs are usually formulated in fractional Sobolev spaces defined on the boundary of the domain. When the boundary is either a closed curve or an open arc, it can be viewed as a Lipschitz manifold, allowing the definition of Sobolev spaces via local charts; see \cite[Chap.~3]{mclean2000strongly}. This approach cannot be directly generalized to multi-arcs, since at branch points the boundary cannot be locally represented as the graph of a function. Even more general frameworks, such as that of \cite{chandler2017sobolev}, are not directly applicable, as they address non-Lipschitz domains of codimension zero.

To the best of our knowledge, the only general approach currently available is due to Claeys and Hiptmair~\cite{Claeys2013}. There, the authors propose a novel construction of Sobolev spaces on multi-arcs based on quotient spaces. Starting from $H^1(\IR^2 \setminus \overline{\Gamma})$ (see Section~2 for a precise definition), they define
$$
\mathbb{H}^{+\half}(\Gamma) \coloneq H^1(\IR^2\setminus \overline{\Gamma}) / H^1_{0,{\Gamma}}(\IR^2),
$$
where $H^1_{0,\Gamma}(\IR^2)$ denotes the subspace of $H^1(\IR^2)$ consisting of functions that vanish in a neighborhood of $\Gamma$. They also introduce
$$
H^{+\half}([\Gamma]) \coloneq H^1(\IR^2) / H^1_{0,{\Gamma}}(\IR^2),
$$
which consists of functions in $\mathbb{H}^{+\half}(\Gamma)$ that coincide on both sides of $\Gamma$, and define $\widetilde{H}^{-\half}([\Gamma])$ as its dual space. Analogous constructions are proposed for $\mathbb{H}^{-\half}(\Gamma)$, $H^{-\half}([\Gamma])$, and $\widetilde{H}^{+\half}([\Gamma])$, based on the range of the operator $\bm{U} \mapsto \bm{n} \cdot \bm{U}|_{\Gamma}$. Within this framework, the usual properties of Sobolev spaces are recovered, integral formulations for the Dirichlet and Neumann problems are derived and shown to be well posed, and numerical results are reported in \cite{averseng2023boundaryelementmethodslaplace}. A key advantage of this construction is its generality: it applies not only to multi-arcs but also to higher-dimensional screen problems with complex intersection patterns.

Another important distinction between the multi-arc and the classical settings concerns the definition of the Neumann trace. In the classical framework, the Neumann trace is defined as an extension of the normal derivative via Green's identities. Although Green's formulas can be extended to the multi-arc case, this naturally leads to an operator describing jumps of normal derivatives, which, while essential for the analysis of integral formulations, does not constitute a satisfactory notion of a Neumann trace.

In \cite{Claeys2013}, this issue is addressed by interpreting the Neumann trace as the canonical surjection associated with the trace operator $\bm{U} \mapsto \bm{n} \cdot \bm{U}|_{\Gamma}$. The resulting analysis closely parallels that of the Dirichlet trace.

Regarding the singular behavior of solutions to the integral formulations on multi-arcs, the only available results appear in \cite{hoskins2020solution}, where the authors analyze singularities arising in transmission problems for composite media near triple junctions. In particular, they prove the presence of singular terms of the form $t^{\beta}$, where $t$ denotes the distance to the triple junction and $\beta$ is the order of the singularity. 

In this work, we address the issues outlined above. Our main contributions can be summarized as follows:
\begin{enumerate}
  \item We introduce a scale of Sobolev spaces for multi-arcs based on classical spaces for open arcs, whose discretization for the Dirichlet problem can be carried out using standard boundary element method implementations.
  \item 
  We define a Neumann trace operator for multi-arcs based on the Neumann trace on open arcs, and analyze it within the newly introduced Sobolev framework.
  \item 
We numerically investigate the singular behavior of solutions to the integral formulation of the Dirichlet problem, showing that the singular exponents are determined in a manner analogous to classical corner singularities.
\item 
We investigate the invertibility of the Neumann integral formulation, and conclude that, in general, it is not possible to discretize the corresponding problem on the classical spaces. 
\end{enumerate}

In contrast to the general approach of \cite{Claeys2013}, our motivation for defining regularity spaces on multi-arcs stems from a simpler observation. Consider the geometry shown in \Cref{fig:NeumanGeo}. If one wishes to define continuity on $\Gamma$ without reference to functions in $\IR^2$, it is not sufficient to require continuity on each individual arc $\Gamma_1$, $\Gamma_2$, and $\Gamma_3$. However, if continuity is imposed on $\Gamma \cap \partial \Omega_1$, $\Gamma \cap \partial \Omega_2$, and $\Gamma \cap \partial \Omega_3$, then a globally continuous function on $\Gamma$ is obtained. This idea naturally extends to Sobolev spaces: we define $H^{\half}(\Gamma)$ as the space of functions that belong simultaneously to $H^{\half}(\Gamma \cap \partial \Omega_1)$, $H^{\half}(\Gamma \cap \partial \Omega_2)$, and $H^{\half}(\Gamma \cap \partial \Omega_3)$. More generally, this yields a family of Sobolev spaces on $\Gamma$ constructed from the corresponding spaces on the individual arcs composing the multi-arc. Besides its conceptual simplicity, this approach allows us to define the full Sobolev scale and to inherit standard properties directly from the theory on open arcs. Furthermore, we will show that there exist some equivalences between the spaces constructed here and the ones from \cite{Claeys2013}.

This construction can be further extended to define the Neumann trace operator, and also the single-layer and double-layer potentials on $\Gamma$ by assembling the corresponding operators on the individual arcs. The resulting operators inherit most of the mapping and analytical properties of the classical layer potentials defined on open arcs, but with some added restrictions regarding the orientation of the normal vectors.

The remainder of the article is organized as follows. \Cref{sec:theoreticalBac} is devoted to the theoretical aspects of boundary integral formulations for multi-arcs. In \Cref{sec:SobSpaces}, we provide a precise definition of a multi-arc and present the construction of the associated Sobolev spaces. \Cref{sec:Traces} addresses the definition of the Dirichlet trace and the Neumann jump. \Cref{sec:DirProb,sec:NeuProb} are devoted to the formulation of the Dirichlet and Neumann problems, respectively; in the latter, we also define the Neumann trace for multi-arcs. \Cref{sec:SingDir} investigates the singularities of the Dirichlet problem. In particular, \Cref{sec:striplejuct} provides a rigorous analysis for a symmetric triple junction, while \Cref{sec:Discretrization,sec:NumRes} describe the discretization scheme and numerical experiments used to study more general configurations. In \Cref{sec:NeumannNumeric} we provide numerical experiments for the Neumann problem, showing that in general this cannot be discretized using classical schemes. Finally, conclusions and our conjecture on the structure of singularities for the Dirichlet problem are presented in \Cref{sec:conclusions}.

\section{Theoretical Background}
\label{sec:theoreticalBac}
\subsection{Preliminaries}

Throughout this work, vectors are denoted through boldface symbols. For any $\bv \in \IR^n$, with $n \in \IN$, we denote the
Euclidean norm $\norm{\bv} = \sqrt{\bv \cdot \bv}$, and the open ball centered at $\bv$ with radius $r >0$,
as $B_{\bv}(r) \coloneq \{ \bw \in \IR^2 : \Vert \bw -\bv \Vert < r \}$. We use the notation $a \lesssim b $ to
indicate that there exists a constant $c>0$, independent of the relevant quantities of the corresponding analysis,
such that $a \leq c b$. Whenever both $a \lesssim b$ and $b \lesssim a$ hold, we write $a \cong b$. 

Given two normed spaces $X$ and $Y$, we write $X \hookrightarrow Y$ if there exists a linear, bounded, injective mapping from
$X$ into $Y$. We write $X\cong Y$ if $X \hookrightarrow Y$ and $Y \hookrightarrow X$. The dual space of $X$ is denoted by $X^\star$.

Let $\Omega\subset \IR^{d}$ be an open and connected set, where $d \in \{1,2\}$. For $m \in \IN_0:=\IN\cup \{0\}$ and $p\in \IN$, we denote by $\C^{m}(\Omega,\IR^{p})$ the space of $\IR^{p}$ valued functions whose derivatives up to order $m$ extend continuously to $\overline{\Omega}$ and, for any $\alpha\in (0,1]$, denote by $\C^{m,\alpha}(\Omega,\IR^{p})$ the space of functions in $\C^{m}(\Omega,\IR^{p})$ whose derivatives of order $m$ are $\alpha$--H{\"o}lder continuous. Furthermore, the space of functions with continuous derivatives up to arbitrary order is denoted by $\C^{\infty}(\Omega, \IR^{p})$ and the subspace of compactly supported functions in $\C^{m}(\Omega,\IR^{p})$ is denoted by $\C_0^{m}(\Omega,\IR^{p})$. Whenever $p =1$, we employ the notation $\mathcal{C}^{m,\alpha}(\Omega)$, $\mathcal{C}^m(\Omega)$, $\mathcal{C}_0^m(\Omega)$. 
The space $\C^\infty_0(\IR^2)$, whose dual is the space of distributions, will be used frequently throughout this work.
The associated duality pairing is denoted as $\langle \cdot , \cdot \rangle_{\IR^2}$. We will also say that a domain
$\Omega \subset \IR^2$ is of class $\C^{m,\alpha}$ if its boundary can be locally parametrized by functions in
$\C^{m,\alpha}(I,\IR^2)$ for an open interval $I\subset\IR$.
The standard space of $\IR$ valued, square integrable functions with respect to the (corresponding) Lebesgue measure on $\Omega$ is
denoted as $L^2(\Omega)$. Moreover, for $s \geq 0 $ we define the Sobolev space $H^s(\Omega)$ as in
\cite[Sec.~3.5]{mclean2000strongly}, i.e., as the space of functions whose weak derivatives up to order
$s$ are in $L^2(\Omega)$ when $s\in\IN$, and through the so-called Sobolev Slobodecki\u{i} semi-norm when $s \notin \IN$ (\emph{c.f.}~\cite[Eq.~3.18]{mclean2000strongly}). We also introduce the local Sobolev spaces as
$$H^s_{\text{\text{loc}}}(\Omega) := \{ F \in C^{\infty}_0(\Omega)^\star: \chi F \in H^s(\Omega), \ \forall \chi \in C^\infty_0(\IR^2) \},$$
and the following subspaces of functions in $H^1(\Omega)$ with integrable Laplacian:
$$
H^1(\Omega; \Delta) \coloneq \{ F \in H^1(\Omega) : \Delta F \in L^2(\Omega)\}, 
$$
$$
H^1_{\text{\text{loc}}}(\Omega; \Delta) \coloneq \{ F \in H_{\text{\text{loc}}}^1(\Omega) : \Delta F \in L^2(\Omega)\}.
$$
Sobolev spaces on the boundary $\partial \Omega$ are defined via localization, as in \cite[Sec.~3.11]{mclean2000strongly}.

Let $I:=(-1,1)$. An \emph{open arc} is defined through a continuous, non-intersecting parametrization $\br:I\to\IR^2$ such that
$\br \in\C^{m,\alpha}(I,\IR^2)$ for $m\in\IN_0$ and $\alpha\in (0,1]$ satisfying
$m+\alpha\geq 1$ and such that $\Vert\br'(t)\Vert \neq 0$ for all $t\ \in I$ (ensuring that $\br$ is injective). We denote by $\C_b^{m,\alpha}$ the set of parametrizations satisfying these
properties and we say that $\Lambda \subset \IR^2$ is an open arc of class $\C^{m,\alpha}$ if there exists
$\br \in \C_b^{m,\alpha}$ such that $\Lambda = \br(I)$. The same convention is used for closed boundaries.

Given a smooth open arc $\Lambda$, we assume that there exists a smooth closed boundary $\widetilde{\Lambda}$ such that
$\Lambda \subset \widetilde{\Lambda}$. For $s \in \IR$, the Sobolev space $H^s(\Lambda)$ is defined as the restriction of elements of $H^s(\widetilde{\Gamma})$. Similarly $\widetilde{H}^s(\Lambda)$ 
is composed of elements whose extension by $0$ is in $H^s(\widetilde{\Gamma})$.
We also recall the duality relation $H^s(\Lambda)^\star \cong \widetilde{H}^{-s}(\Lambda)$, where the duality pairing is an
extension of the $L^2(\Lambda) \equiv H^0(\Lambda)$ inner product, and is denoted as $\langle \cdot , \cdot\rangle_{\Lambda}$. For
further details on the construction and definition of Sobolev spaces, we refer to \cite[Chap.~3]{mclean2000strongly}.

\subsection{Sobolev Spaces}
\label{sec:SobSpaces}

In this section, we introduce the Sobolev spaces on multi-arc geometries. We begin by providing a precise definition of a multi-arc.
\begin{definition}[Multi-arc]
  \label{def:multiarc}
  We say that a bounded set $\Gamma\subset\IR^2$ is a multi-arc of class $\C^{m,\alpha}$ for some $m\in\IN_0$ and
  $\alpha \in (0,1]$ satisfying $m+\alpha \geq 1$, if the following conditions are met:
  \begin{enumerate}
  \item There is some $n\in\IN$ such that $\overline\Gamma=\bigcup_{i=1}^n\overline\Gamma_i$, where $\{\Gamma_i\}_{i=1}^n$ is a pairwise
    disjoint collection of open arcs of class $\C^{m,\alpha}$ such that their closures (i.e., $\{\overline\Gamma_i\}_{i=1}^n$) may intersect only at their endpoints.
  \item There exists a pairwise disjoint collection of bounded open sets of class $\C^{0,1}$, denoted $\{\Omega_i\}_{i=1}^n$, such that
    $\partial \Omega_i \cap \partial \Omega_j \subset \overline{\Gamma}$ for all $i,\ j\in\{1,\hdots,n\}$ with $i\neq j$, each open arc in the decomposition $\{\Gamma_i\}_{i=1}^n$
    is contained in exactly two of the boundaries $\{\partial\Omega_i\}_{i=1}^n$, and
    $$\Omega := \mathrm{int}\left(\bigcup\limits_{i=1}^n\overline\Omega_i\right)$$ is a $\C^{0,1}$ domain.
    \item 
There exists some $\epsilon_0 > 0$ such that for every $\epsilon\in(0,\epsilon_0)$ we have that 
$$
\Gamma_\epsilon \coloneq \{ \bx  \in \IR^2 : \ \text{dist}(\bx, \Gamma) < \epsilon \} 
$$
is an open $\mathcal{C}^{0,1}$ domain.    
 \end{enumerate}
 
In addition to the decomposition in disjoint open arcs $\{\Gamma_i\}_{i=1}^n$, we also introduce the decomposition
$$\overline\Gamma = \bigcup_{i=1}^n\overline{\widehat{\Gamma}_i},$$ where $\widehat{\Gamma}_i = \Gamma \cap \partial \Omega_i$ for $i \in \{1,\hdots,n\}$. 

Points belonging to the intersection of two or more arcs are called \emph{junction points}.
  If exactly two arcs meet at such a point, we refer to it as a \emph{corner}.
  If three or more arcs meet, the point is referred to as a \emph{branch point} or \emph{multiple junction}.
\end{definition}
A graphical illustration of the configuration in \Cref{def:multiarc} is shown in \Cref{fig:NeumanGeo}.
From \Cref{def:multiarc} it is easy to deduce that for any $i, j,\ell\in\{1,\hdots,n\}$ such that
$j\neq\ell$ and $\partial \Omega_j$ and $\partial \Omega_{\ell}$ contain the open arc $\Gamma_i$, then the
exterior unit normal vectors in $\Omega_j$ and $\Omega_\ell$ satisfy
$$(\bm{n}_{\partial \Omega_j})|_{\Gamma_i} = - (\bm{n}_{\partial \Omega_{\ell}})|_{\Gamma_i}.$$
From here onward, let $\Gamma\subset\IR^2$ denote a given multi-arc
composed of $n\in\IN$ disjoint open arcs $\{\Gamma_i\}_{i=1}^n$. 
The space $L^2(\Gamma)$ is defined as the set
of measurable functions on $\Gamma$ such that
$$\int\limits_\Gamma \abs{f(\bx)}^2\ \d\bx = \sum_{i=1}^n \int\limits_{\Gamma_i} \abs{f(\bx)}^2\ \d\bx < \infty,$$
which can be identified with its dual through the duality pairing
\begin{align}
\label{eq:L2duality}
  \dup{f}{g}_\Gamma \coloneq \int\limits_\Gamma f(\bx) g(\bx)\ \d\bx = \half \sum_{i=1}^n \dup{f}{g}_{\widehat{\Gamma}_i}.
\end{align}

The space of smooth functions on $\Gamma$ is defined by restricting functions in $\C_0^\infty(\IR^2)$ to $\Gamma$, and will be
denoted as $\C^\infty(\Gamma)$. We also define the space of smooth functions with compact support on $\Gamma$ as 
$$ \mathcal{D}(\Gamma) :=
\left\lbrace u \in \C^\infty(\Gamma) : \supp(u|_{\widehat{\Gamma}_i}) {\Subset} \widehat{\Gamma}_i
  \text{ for } i\in\{1,\hdots , n\} \right\rbrace. 
$$

Although $\mathcal{D}(\Gamma)$ and $\C^\infty(\Gamma)$ are composed of the most regular functions that can be defined on $\Gamma$, these are not $\C^\infty$ in the usual manifold sense. For $s\in[-1,1]$, we introduce the product spaces
$$
H^s_\times(\Gamma) := \prod_{i=1}^n H^s(\widehat{\Gamma}_i)\quad\text{and}\quad 
\widetilde{H}^s_\times(\Gamma) := \prod_{i=1}^n \widetilde{H}^s(\widehat{\Gamma}_i).
$$
As a consequence of the duality relation for open arcs, we have that
$H^s_\times(\Gamma)^\star \cong \widetilde{H}^{-s}_\times(\Gamma)$. We denote the respective duality pairing as
$\dup{\cdot}{\cdot}_{H^s_\times(\Gamma), \widetilde{H}^{-s}_\times(\Gamma)}$, which is an extension of the
inner product on the space
$$\prod_{i=1}^n L^2(\widehat{\Gamma}_i).$$ Furthermore, we notice
that $\C^\infty(\Gamma)$ can be identified as a subset of $H^s_\times(\Gamma)$ through the injection
$$u \mapsto (u|_{\widehat{\Gamma}_1}, \hdots , u|_{\widehat{\Gamma}_n}).$$ 
An analogous injection enables us to identify $\mathcal{D}(\Gamma)$ with a subset of $\widetilde{H}^s_{\times}(\Gamma)$.
We then define, for $s\in[-1,1]$, the following closed subspaces: 
\begin{align}
\label{eq:hsdefi}
H^s(\Gamma) := \overline{\C^\infty(\Gamma)}^{H^s_\times(\Gamma)}\quad\text{and}\quad
\widetilde{H}^s(\Gamma) := \overline{\mathcal{D}(\Gamma)}^{\widetilde{H}^s_\times(\Gamma)}. 
\end{align}
It follows directly from the construction of the spaces above, that
\begin{align}
\label{eq:norms}
  &\norm{u}_{H^s(\Gamma)}^2 = \sum_{i=1}^n \norm{u|_{\widehat{\Gamma}_i}}_{H^s(\widehat{\Gamma}_i)}^2 = \sum_{i=1}^n \inf_{\substack{\hat{u}_i \in H^s(\partial \Omega_i)\\ \hat{u}_i|_{\widehat{\Gamma}_i} = u|_{\widehat{\Gamma}_i}} }\norm{\hat{u}_i}^2_{H^s( \partial \Omega_i)},\\
  &\norm{u}_{\widetilde{H}^s(\Gamma)}^2 = \sum_{i=1}^n \norm{u\vert_{\widehat{\Gamma}_i}}^2_{\widetilde{H}^s(\widehat{\Gamma}_i)} = \sum_{i=1}^n \norm{\tilde{u}_i}^2_{H^s(\partial \Omega_i)},
\end{align}
where $\tilde{u}_i$ denotes the extension by zero of $u|_{\widehat{\Gamma}_i}$.
We can easily see from the previous definition, that both spaces are equivalent to $L^2(\Gamma)$ when $s=0$.

In contrast to the product spaces, the duality relation
$H^s_\times(\Gamma)^\star\cong\widetilde H^{-s}_\times(\Gamma)$ does not extend to $H^s(\Gamma)$ and
$\widetilde H^s(\Gamma)$. To see this, take the case $n = 3$ and consider $F = (1,2,3) \in H^s_\times(\Gamma)$
for every $s \in [-1,1]$. Then, $F$ defines both a bounded linear functional on $\widetilde{H}^{-s}_\times(\Gamma)$ and
a bounded linear functional in $\widetilde{H}^{-s}(\Gamma)$. In particular, for any $\varphi \in \mathcal{D}(\Gamma)$,
we have that 
$$ \dup{F}{\varphi}_{H^s_\times(\Gamma), \widetilde{H}^{-s}_\times(\Gamma)}
= \int\limits_{\widehat{\Gamma}_1} \varphi(\bx)\ \d\bx
+ \int\limits_{\widehat{\Gamma}_2} 2 \varphi(\bx)\ \d\bx
+ \int\limits_{\widehat{\Gamma}_3} 3 \varphi(\bx)\ \d\bx
= \int\limits_{\Gamma} f(\bx) \varphi(\bx)\ \d\bx,$$
where 
$$f(\bx) = \begin{cases}
  3 \quad \bx \in \partial \Omega_1 \cap \partial \Omega_2 \\
  5 \quad \bx \in \partial \Omega_2 \cap \partial \Omega_3 \\
  4 \quad \bx \in \partial \Omega_1 \cap \partial \Omega_3
\end{cases}$$
If the relation $H^{s}(\Gamma)^\star \cong \widetilde{H}^{-s}(\Gamma)$ were true, then we should have that
$f \in H^s(\Gamma)$ for every $s \in [-1,1]$. However, this does not hold in general due to the discontinuity of $f$
across $\widehat{\Gamma}_1$, $\widehat{\Gamma}_2$, and $\widehat{\Gamma}_3$.

We can characterize the dual spaces of $H^s(\Gamma)$ and $\widetilde{H}^s(\Gamma)$ using \cite[Lemma~2.2]{chandler2017sobolev} or \cite[Thm.~4.9]{rudin1991functional}.
Using these results, we obtain the following relations:
\begin{align*}
\widetilde{H}^{-s}(\Gamma) \subset H^s(\Gamma)^\star \subset \widetilde{H}^{-s}_\times(\Gamma), \\
{H}^{-s}(\Gamma) \subset \widetilde{H}^s(\Gamma)^\star \subset {H}^{-s}_\times(\Gamma), 
\end{align*}
where we have abused the notation to denote $H^s(\Gamma)^\star$ (respectively $\widetilde{H}^s(\Gamma)^\star$)
as the identification of the dual space of $H^s(\Gamma)$ (respectively $\widetilde{H}^s(\Gamma)$), obtained from \cite[Lmm.~2.2]{chandler2017sobolev} or \cite[Thm.~4.9]{rudin1991functional} using the identification
of the dual of $H^s_\times(\Gamma)$ with $\widetilde{H}^{-s}_\times(\Gamma)$ for every $s \in [-1,1]$.
Furthermore, notice that for $f,g \in L^2(\Gamma)$, \cref{eq:L2duality} yields
$$
\dup{f}{g}_{H^s_\times(\Gamma), \widetilde{H}^{-s}_\times(\Gamma)} = \half \dup{f}{g}_\Gamma.
$$
Hence, the bilinear form $\dup{\cdot}{\cdot}_\Gamma$ can be extended to a continuous bilinear form
with its second argument in $H^s(\Gamma)$ or $\widetilde{H}^s(\Gamma)$ and the first one in the
respective dual space. Moreover, since the realization of the duals is unitary, we have that for
any $f \in H^s(\Gamma)^\star$ and $g \in \widetilde{H}^s(\Gamma)^\star$,
\begin{align}  
\label{eq:dualunitary}
  \norm{f}_{\widetilde{H}^{-s}_\times(\Gamma)}
  = \half \sup_{\varphi \in H^s(\Gamma)} \frac{ \dup{f}{\varphi}_\Gamma}{\norm{\varphi}_{H^s(\Gamma)}}, \quad 
  \norm{g}_{H^{-s}_\times(\Gamma)}
  = \half \sup_{\phi \in \widetilde{H}^s(\Gamma)} \frac{ \dup{g}{\phi}_\Gamma}{\norm{\phi}_{\widetilde{H}^s(\Gamma)}},
  \quad \forall s \in [-1,1].
\end{align}
\begin{remark}
  The spaces introduced here are closely related to those in \cite{Claeys2013}. In particular, the spaces
  $H^s_\times(\Gamma)$, $\widetilde{H}^s_\times(\Gamma)$ are multi-trace spaces, in the sense that a function has two
  values in each point of $\Gamma$. Furthermore, arguing as in the case of standard screens presented in \cite[Sec.~5.2]{Claeys2013} we obtain
\begin{align*}
  \mathbb{H}^{+\half}(\Gamma)\hookrightarrow H^{\half}_\times(\Gamma),
  \quad \text{and} \quad
  \mathbb{H}^{-\half}(\Gamma) \hookrightarrow H^{-\half}_\times(\Gamma), 
\end{align*}
where $\mathbb{H}^{+\half}(\Gamma)$ and $\mathbb{H}^{-\half}(\Gamma)$ are the multi-trace spaces as defined in \cite{Claeys2013}.
Moreover, from the density of $\C^\infty_0(\IR^2)$ in $H^1(\IR^2)$ and \Cref{lem:equiv} we can rigorously show that
$$
H^{+\half}([\Gamma]) \cong H^{\half}(\Gamma), 
$$
where $H^{+\half}([\Gamma])$ is the single trace space introduced in \cite[Sec.~6.1]{Claeys2013}. Finally, by definition,
the jump spaces in \cite[Sec.~6.4]{Claeys2013} are equivalent to $H^{\half}(\Gamma)^\star$. 
\end{remark}

\subsection{Trace Operators on Multi-arcs}
\label{sec:Traces}

\subsubsection{Dirichlet Trace}
Let $u \in \C_0^\infty(\IR^2)$ and let $s \in (\half,\threehalf)$. By  the standard properties of the trace space, see
\cite[Thm.~3.38]{mclean2000strongly}, and the definition of the norms for spaces on open arcs, we have that for every $i\in\{1,\hdots, n\}$ it holds that
\begin{align}
\label{eq:dirTraceArc}
\Vert u\Vert_{\widehat{\Gamma}_i}\Vert_{H^{s-\half}(\widehat{\Gamma}_i) } \lesssim 
\Vert u\Vert_{H^s(\IR^2)}.
\end{align}
\Cref{eq:dirTraceArc}, together with the definition of $H^{s-\half}(\Gamma)$ in \cref{eq:hsdefi}, yields
$$
\Vert u|_{\Gamma }\Vert_{H^{s-\half}(\Gamma) } \lesssim 
 \Vert u\Vert_{H^s(\IR^2)} .
$$ 
This estimate allows us to define the \emph{Dirichlet trace operator},
\[
\gamma_D : H^s(\IR^2) \longrightarrow H^{s-\half}(\Gamma),
\]
as the continuous extension of the restriction operator. It can also be shown that the  Dirichlet 
trace operator can be defined, as a bounded operator, on local spaces. In what follows, we denote the Dirichlet trace operator on $\Gamma$ as $$\gamma_D:H^s_{\mathrm{\text{loc}}}(\IR^2)\to H^{s-\half}(\Gamma),$$ while the
classical Dirichlet trace operator on $\partial\Omega_i$ is denoted as $$\gamma_D^i:H^{s}(\Omega_i)\to H^{s-\half}(\partial\Omega_i),$$ for each $i\in\{1, \hdots, n\}$.

\subsubsection{Neumann Jump}

We now seek to extend the definition of the Neumann jump across the multi-arc $\Gamma$, which we carry out through an application Green's integration by parts formula. Before proceeding, we introduce an assumption that facilitates its application to our context.

\begin{assumption}
\label{assumption:RGamas}
 Let $\IR_\Gamma := \{ r \in \IR : \overline{\bigcup_{i=1}^n \Omega_i} \subset B_{\bm{0}}(r) \}$.
 For every $r \in \IR_\Gamma$ there exists a family of open Lipschitz domains $\{\Omega_i^r\}_{i=1}^n$
 (where $n\in\IN$ is as in \Cref{def:multiarc}) such that for any $r, p\in\IR_\Gamma$ and $i,j\in\{1,\hdots,n\}$ such that
 $r < p$ and $i\neq j$, it holds that
 \begin{align*}
  \bigcup_{i=1}^n \overline{\Omega_i^r} = \overline{B_{\bm{0}}(r)}, \quad \Omega_i^r \cap \Omega_j^{r} = \emptyset,\quad \Omega_i \subset \Omega_i^{r}\quad\text{and}\quad\Omega_i^r \subset \Omega_i^{p}.
 \end{align*}
\end{assumption}
Throughout the rest of the article, we will assume that the previous assumption holds.

Let $U$ be a piecewise smooth function in $\Omega_i^r$ for each $i\in\{1,\hdots,n\}$. Then, we define its outward normal derivative on $\partial \Omega_i^r$, for each $i\in\{1,\hdots,n\}$, by
\[
\partial_{\bm n}^i U := \bm n_i \cdot \gamma_D^i(\nabla U),
\]
where $\bm n_i$ denotes the unit normal vector to $\Omega_i$, oriented so as to point to the exterior of $\Omega^r_i$.
 If $U \in H^1 (\Omega_i ^r; \Delta)$, we denote by $\gamma_N^iU \in H^{-\half}(\partial \Omega_i^r)$ its Neumann trace, which is an extension of the normal derivative, see \cite[Chapter 4]{mclean2000strongly}. 
Notice that we commit a slight abuse of notation, as there is no dependence on $r\in\IR_\Gamma$ in the notation of the Neumann trace, which will always be clear from the context. We will also reuse the notation $\gamma_D^i$ to denote the Dirichlet trace in the domain $\Omega_i^r$, regardless of the value of $r\in\IR_\Gamma$.

We now begin our work to extend the notion of the Neumann jump across multi-arc geometries. Let $U,V$ be compactly supported functions in $\IR^2$ such that they are smooth inside $\Omega_i^r$ for each $i\in\{1,\hdots,n\}$ and $r \in \IR_\Gamma$. Using Green's formula, we have that
\begin{align*}
  \sum_{i=1}^n \int\limits_{\Omega_i^r} \Delta U(\bx) V(\bx) +\nabla U(\bx) \cdot \nabla V(\bx)\ \d\bm{x}&= \sum_{i=1}^n \langle \partial_{\bm{n}}^i U, \gamma_D^i V \rangle_{ \partial \Omega^r_i}\\
  &= \sum_{i=1}^n \langle \partial_{\bm{n}}^i U, \gamma_D^i V \rangle_{\widehat{\Gamma}_i} + \langle \partial_{\bm{n}}^i U, \gamma_D^i V \rangle_{\partial\Omega_i^r\setminus \widehat{\Gamma}_i}\\
  &= \sum_{i=1}^n -\langle [\partial_{\bm{n}}^i U], \gamma_D^i V \rangle_{\widehat{\Gamma}_i} + \langle \partial_{\bm{n}}^i U, \gamma_D^i V \rangle_{ \partial \Omega_i^r\setminus \widehat{\Gamma}_i},
\end{align*}
where $[\partial_{\bm{n}}^i U]$ corresponds to the jump of the normal derivative of $U$ across $\widehat{\Gamma}_i$ for each $i \in \{1,\hdots,n\}$.
Then, if $V$ is globally continuous, $U$ is sufficiently smooth across $\IR^2\setminus\overline{\Gamma}$, and $r\in\IR_\Gamma$ is large enough, we have that $\langle\partial_{\bm{n}}^i U, \gamma_D^i V \rangle_{ \partial \Omega_i^r\setminus \widehat{\Gamma}_i}=0$ for each $i\in\{1,\hdots, n\}$ and, therefore,

\begin{align}
\label{eq:neujumpsmooth}
\sum_{i=1}^n \int\limits_{\Omega_i^r} \Delta U(\bx) V(\bx) +\nabla U(\bx) \cdot \nabla V(\bx)\ \d\bm{x} = -\sum_{i=1}^n \langle [\partial_{\bm{n}}^i U], \gamma_D V \rangle_{\widehat{\Gamma}_i}=: \langle [\partial_{\bm{n}}U],\gamma_D V\rangle_\Gamma.
\end{align}
The latter equation enables us to define the Neumann jump of $U$ acting on $\gamma_D V$. In what follows, we present some technical results that allows us to properly extend the definition in \cref{eq:neujumpsmooth} to the adequate Sobolev space setting.
First, we recall the relevant properties of Burekov's extension operator in our context \cite{burenkov1998sobolev,extensions17}. This operator will be used later on to establish a relevant density result. 

\begin{lemma} \label{lemma:extension}
    There exists a bounded, continuous, and linear operator $\mathcal{E}: H^1(\Omega) \rightarrow H^1(\IR^2)$ such that
    \begin{enumerate}
        \item $\mathcal{E} U |_{\Omega}=  U$ for every $U \in H^1(\Omega)$.
        \item $\mathcal{E}$ maps $\mathcal{C}^\infty(\Omega)$ to $\mathcal{C}^2(\IR^2)$. 
        \item If $U$ is zero on a relative open neighborhood of $\Gamma$ in $\Omega$, then $\mathcal{E}U$ is zero on a open neighborhood of $\Gamma$ in $\IR^2$.  
    \end{enumerate}
\end{lemma}
\begin{proof}
    The first and second points are standard properties of Burenkov's extension operator, and can be found in \cite[Chap.~6]{burenkov1998sobolev} and \cite{extensions17}. The third point is an immediate consequence of the construction of the operator, as the value at a point $\bx \in \IR^2 \setminus \overline{\Omega}$ depends on the values on some appropriately reflected points inside $\Omega$. Once again, we refer to \cite[Chap.~6]{burenkov1998sobolev} and \cite{extensions17} for further details on its construction. (This local behavior of the operator was also explicitly mentioned in \cite[Sec.~1]{extensions17}).
\end{proof}

\begin{lemma}
\label{lemma:dens}
    Fix $r\in\IR_\Gamma$.
    For any $U \in \mathcal{C}_0^\infty(B_{\bm{0}}(r))$ such that $\gamma_D U = 0$ in $\Gamma$, there exists a sequence $\{U_k\}_{k \in \IN} \subset \mathcal{C}^\infty(\Omega)$ such that each $U_k$ is zero on an open neighborhood of $\Gamma$ relative to $\Omega$, and such that
        $$
    \lim_{k \rightarrow \infty}\Vert U_k - U \Vert_{H^1(\Omega)} =  0. 
    $$
    Furthermore, there exist extensions $\widehat{U}_k\in\mathcal{C}^2_0(\IR^2 \setminus \overline{\Gamma})$ to each $U_k\in\C^\infty(\Omega)$ such that
    $\supp{\widehat{U}_k}\subset B_{\bm{0}}(r)$, as well as an extension
    $\widehat{U}\in\C^2_0(\IR^2 \setminus \overline{\Gamma})$ of $U$, also supported in $B_{\bm{0}}(r)$, such that 
    $$
    \lim_{k \rightarrow \infty}\Vert \widehat{U}_k - \widehat{U} \Vert_{H^1(\IR)} =  0.
    $$
\end{lemma}

\begin{proof}
For any open arc $\Lambda\subset\IR^2$, the space $\mathcal{C}^\infty_0(\IR^2 \setminus \overline{\Lambda})$ is dense in $\{ F \in H^1(\IR^2) : \gamma_{D,\Lambda} U = 0\}$, where $\gamma_{D,\Lambda}$ is the Dirichlet trace operator on $\Lambda$. We refer to Appendix \ref{app:profDensOpenArc} for a detailed proof of this statement.

Take $r\in \IR_\Gamma$ and $U\in \mathcal{C}_0^\infty(B_{\bm{0}}(r))$ with $\gamma_D U  = 0$. Then, for each $i\in\{1,\hdots,n\}$, we have that $\gamma_{D,\widehat{\Gamma}_i} U = 0$ and we can construct a sequence $\{U_k^{(i)} \}_{k \in \IN}  \subset \mathcal{C}^\infty_0 (\IR^2 \setminus \overline{\widehat{\Gamma}_i})$ converging to $U$ in $H^1(\IR^2)$. Then, for each $k\in\IN$
we define $U_k\in H^1(\Omega)$ as
$$
U_k(\bx) := \begin{cases}
U_k^{(1)}(\bx)\quad&\text{if}\quad \bx\in\Omega_1\\
    \quad\vdots\quad &\vdots\quad\quad\vdots\\
U_k^{(n)}(\bx)\quad&\text{if}\quad \bx\in\Omega_n
\end{cases}
$$
so that $U_k\in\mathcal{C}^\infty(\Omega)$ and $U_k\equiv 0$ in an open neighborhood of $\Gamma$ in $\Omega$. Moreover, since
for each $i\in\{1,\hdots,n\}$ we have that $$U_k^{(i)} \xrightarrow[k\to\infty]{} U$$ in $H^1(\IR^2)$, it follows that
$\{U_k\}_{k\in\IN}$ converges to $U$ in $H^1(\Omega)$.
  
Consider now a fixed $\chi \in \mathcal{C}_0^\infty(\IR^2)$ such that $\chi = 1$ in $\supp U\cup\Omega$ and $ \supp \chi \subset B_{\bm{0}}(r)$. Then, we define 
$$
\widehat{U}:=\chi \mathcal{E} ( U)\quad\text{and}\quad\widehat{U}_k  := \chi \mathcal{E} ( U_k)\quad \forall k \in \IN, 
$$
where $\mathcal{E}$ is the Burekov's extension operator. From \Cref{lemma:extension}, it is clear that $\widehat{U}_k$ is an extension of $U_k$, as is that for every $k \in \IN$ there exists a neighborhood of $\Gamma$ in $\IR^2$ such that $\widehat{U}_k$ is $0$ in that neighborhood. From this it follows that
\begin{equation*}
\begin{split}
    \Vert \widehat{U}_k  - \widehat{U}\Vert_{H^1(\IR^2)} &= \Vert \chi \mathcal{E}(U_k) - \chi \mathcal{E}(U)\Vert_{H^1(\IR^2)} \\
    &\lesssim    \Vert  \mathcal{E}(U_k) -  \mathcal{E}(U)\Vert_{H^1(\IR^2)} \\
    & \lesssim \Vert  U_k -  U\Vert_{H^1(\Omega)}. 
    \end{split}
\end{equation*}
which proves that $\{\widehat{U}_k\}_{k=1}^{\infty}$ converges to $\widehat{U}$ on $H^1(\IR^2)$.
\end{proof}



The following lemma, concerning an equivalent norm in $H^{\half}(\Gamma)$, will be instrumental in
showing that the notion of the \emph{Neumann jump} is well defined for a certain class of functions.
\begin{lemma}
\label{lem:equiv}
Define, for all $u\in H^{\half}(\Gamma)$, the following quantity:
\begin{align*}
\vvvert u\vvvert_{H^{\half}(\Gamma)}:=\inf
\left\{\norm{U}_{H^1(\IR^2)} : U\in H^1(\IR^2)\ \text{such that}\ \gamma_D^i(U\vert_{\Omega_i})=u\ \text{in}\ \widehat{\Gamma}_i\ \forall i \in\{1,\hdots,n\}\right\}.
\end{align*}
Then, $\vvvert \cdot\vvvert_{H^{\half}(\Gamma)}$ is an equivalent
norm in $H^{\half}(\Gamma)$.
\end{lemma}
\begin{proof}
 Let $u\in C^\infty(\Gamma)$ be arbitrary. On each open arc $\widehat{\Gamma}_i$ the $H^{\frac{1}{2}}(\widehat{\Gamma}_i)$ norm is equivalent to the infimum of the $H^1-$norms over all $H^1(\Omega_i)$ possibles extensions, for $i \in \{1,\hdots,n\}$, see \cite[Thm.~2.22]{STEI07}. Hence, we obtain
 \begin{align}
  \norm{u}_{H^{\half}(\Gamma)}^2
  &=\sum\limits_{i=1}^n\norm{u}_{H^{\half}(\widehat{\Gamma}_i)}^2 \nno \\
  &\cong\sum\limits_{i=1}^n\inf
   \left\{\norm{U}_{H^1(\Omega_i)}^2 : U\in H^1(\Omega_i)\ \text{such that}\ \gamma_D^iU=u\ \text{in}\ \widehat{\Gamma}_i\right\}\nno \\
  &=\inf\left\{\sum\limits_{i=1}^n\norm{U_i}_{H^1(\Omega_i)}^2 : U_i\in H^1(\Omega_i)\ \text{such that}\ \gamma_D^iU_i=u\ \text{in}\ \widehat{\Gamma}_i\ \forall i\in\{1,\hdots,n\}\right\}\nno\\
  \label{eq:lem:h1split}
  &=\inf\left\{\norm{U}_{H^1(\Omega)}^2 : U\in H^1(\Omega)\ \text{such that}\ \gamma_D^iU\vert_{\Omega_i}=u\ \text{in}\ \widehat{\Gamma}_i\ \forall i\in\{1,\hdots,n\}\right\}\\
  \label{eq:lem:h1rnnorm}
  &=\inf\left\{\norm{U}_{H^1(\IR^2)}^2 : U\in H^1(\IR^2)\ \text{such that}\ \gamma_D^iU\vert_{\Omega_i}=u\ \text{in}\ \widehat{\Gamma}_i\ \forall i\in\{1,\hdots,n\}\right\}
 \end{align}
 where \cref{eq:lem:h1split} follows from the fact that if $U\in L^2(\Omega)$ is such that
 $U\vert_{\Omega_i}\in H^1(\Omega_i)$ for each $i\in\{1,\hdots, n\}$ and
$\gamma_D^iU\vert_{\partial\Omega_i}=\gamma_{D}^jU\vert_{\partial\Omega_j}$ in $\partial\Omega_i\cap\partial\Omega_j$
 for all $i,j\in\{1,\hdots,n\}$, then $U\in H^1(\Omega)$. \Cref{eq:lem:h1rnnorm}, on the other hand, follows from the definition of the norm of $H^1(\Omega)$ (\emph{c.f.}~\cite[Chap.~3]{mclean2000strongly}), and corresponds to the desired result.
\end{proof}

\begin{proposition}
\label{prop:boundNeumanJump}
Let $U \in H^1_{\text{\text{loc}}}(\Delta; \IR^2 \setminus \overline \Gamma)$, $v\in\C^\infty(\Gamma)$ and $V\in\C_0^{\infty}(\IR^2)$ such that it extends $v$ to $\IR^2$. Then, for any $r_0 \in\IR_\Gamma$ it holds that
\begin{align*}
\abs{\lim_{r \rightarrow \infty}\sum_{i=1}^n \langle \gamma_N^i U, \gamma_D^i V \rangle_{ \partial \Omega^r_i}}\leq C_{r_0}\Vert v\Vert_{H^{\half}(\Gamma)}\sum_{i=1}^n \left(\norm{\Delta U}_{L^2(\Omega_i^{r_0})}+\snorm{ U}_{H^1(\Omega_i^{r_0})} \right)
,
\end{align*}
where the positive constant $C_{r_0}$ depends on the chosen $r_0\in\IR_\Gamma$.
Furthermore, the quantity $$\lim\limits_{r \rightarrow \infty}\sum_{i=1}^n \langle \gamma_N^i U, \gamma_D^i V \rangle_{ \partial \Omega^r_i},$$ does not depend on the chosen extension $V$ of $v\in\C^\infty(\Gamma)$.
\end{proposition}
\begin{proof}
 Let $U \in H^1_{\text{\text{loc}}}(\IR^2 \setminus \overline \Gamma; \Delta)$ and consider, first, $V\in\C^\infty_0(\IR^2\setminus\overline{\Gamma})$ and $r\in\IR_\Gamma$ such that
 $\supp(V)\subset B_{\bm{0}}(r)$. Then, we have that
 \begin{align*}
  \int\limits_{\IR^2\setminus\overline \Gamma} \Delta U(\bx)V(\bx)\ \d\bx
  &=\int\limits_{B_{\bm{0}}(r)\setminus \overline{\Gamma}} \Delta U(\bx)V(\bx)\ \d\bx\\
  &=\sum\limits_{i=1}^n\int\limits_{\Omega_i^r} \Delta U(\bx)V(\bx)\ \d\bx\\
  &= \sum\limits_{i=1}^n -\int\limits_{\Omega_i^r} \nabla U(\bx)\cdot \nabla V(\bx)\ \d\bx +
  \dup{\gamma_N^i U}{\gamma_D^i V}_{\partial\Omega_i^r}.
 \end{align*}
 Analogously, it holds that
 \begin{align*}
  \int\limits_{\IR^2\setminus \overline \Gamma} \Delta V(\bx) U(\bx)\ \d\bx
  = \sum\limits_{i=1}^n -\int\limits_{\Omega_i^r} \nabla V(\bx)\cdot \nabla U(\bx)\ \d\bx
  + \dup{\gamma_N^i V}{\gamma_D^i U}_{\partial\Omega_i^r},
 \end{align*}
 and, by definition of the weak derivatives, that
 $$\int\limits_{\IR^2\setminus \overline \Gamma} \Delta U(\bx) V(\bx)\ \d\bx = \int\limits_{\IR^2\setminus\overline \Gamma} \Delta V(\bx) U(\bx)\ \d\bx,$$
so that by combining the above computations it follows \ that
$$\sum\limits_{i=1}^n\dup{\gamma_N^i U}{\gamma_D^i V}_{\partial\Omega_i^r} = \sum\limits_{i=1}^n\dup{\gamma_N^i V}{\gamma_D^i U}_{\partial\Omega_i^r}.$$
Moreover, since $\gamma_D^iU\in H^{\frac{1}{2}}(\partial\Omega_i^r)$ and $V\in\C_0^\infty(\IR^2\setminus\overline{\Gamma})$, it holds that
$$
\sum\limits_{i=1}^n\dup{\gamma_N^i U}{\gamma_D^i V}_{\partial\Omega_i^r} = 
\sum\limits_{i=1}^n
\int\limits_{\partial \Omega_i^r} \partial_{\bm{n}}^i V(\bx) \gamma_D^i U(\bx) \ \d\bx,
$$
and recalling that $V=0$ on a neighborhood of $\Gamma$ and of $\partial B_{\bm{0}}(r)$ yields
\begin{align}
\label{eq:nullsum}
\sum\limits_{i=1}^n
\int\limits_{\partial \Omega_i^r} \partial_{\bm{n}}^i V(\bx) \gamma_D^i U(\bx) \ \d\bx=
\sum\limits_{i=1}^n
\int\limits_{\partial \Omega_i^r \setminus (\overline {\Gamma} \cup\partial B_{\bm{0}}(r)) } \partial_{\bm{n}}^i V(\bx) \gamma_D^i U(\bx) \ \d\bx.
\end{align}
However, since $U \in H^1_{\loc}(\IR^2 \setminus \overline{\Gamma})$, it follows that the jump of the Dirichlet trace of $U$ across interfaces of the form $$\partial\Omega_i^r\setminus\Gamma\quad\forall\ i\in\{1,\hdots,n\}$$ is $0$, from where it follows that the sum in \cref{eq:nullsum} must also be equal to $0$ and, therefore,
\begin{align}
\label{eq:prop:intRes1}
\sum\limits_{i=1}^n\dup{\gamma_N^i U}{\gamma_D^i V}_{\partial\Omega_i^r}=0.
\end{align}

We now consider $V \in \mathcal{C}^2_0(\IR^2 \setminus \overline{\Gamma})$ with support in $B_0(r)$, for which there exists an open neighborhood of $\Gamma$ where $V(\bx)=0$ and, therefore, there also exists some $\epsilon >0$ such that $V(\bx)=0$ for all $\bx\in\Gamma_\epsilon$, where $\Gamma_\epsilon$ is as in 
\Cref{def:multiarc}. Then, there exists a sequence $\{V_k\}_{k \in \IN} \subset \mathcal{C}^\infty_0(B_0(r) \setminus \overline{\Gamma_\epsilon})$ such that $\{V_k\}_{k\in\IN}$ converges to $V$ in $H^1(\IR^2)$ (see \cite[Thm.~3.40]{mclean2000strongly}). It follows by our previous result in \cref{eq:prop:intRes1} that
$$
\sum_{i=1}^n \langle \gamma_{N}^i U, \gamma_D^i V_k\rangle_{\partial \Omega_i^r} = 0 \quad \forall k \in \IN,  
$$
which directly implies that 
$$
\sum_{i=1}^n \langle \gamma_{N}^i U, \gamma_D^i V\rangle_{\partial \Omega_i^r} = 0.
$$
 
Now consider $V,W\in \C^\infty_0(\IR^2)$ such that both extend a given $v\in\C^\infty(\Gamma)$, and take $r\in\IR_\Gamma$ such that both
$V$ and $W$ are supported in $B_0{(r)}$. Then, by \Cref{lemma:dens} and since $\gamma_D(V-W)=0$, there exists a sequence $\{D_k\}_{k \in \IN} \subset \C^2_0(\IR^2 \setminus \overline \Gamma )$ and some $D \in H^1(\IR^2)$ with support in $B_{\bm{0}}(r)$, such that $D_k$ converges to $V-W$ in $H^1(\Omega)$ and to $D$ in $H^1(\IR^2)$. Then, it holds that
$$
\sum_{i=1}^n \langle \gamma_N^i U , \gamma_D^i (V- W) \rangle_{\partial \Omega_i^r } = 
\sum_{i=1}^n \langle \gamma_N^i U , \gamma_D^i (V- W - D_k) \rangle_{\partial \Omega_i^r }  +  
\sum_{i=1}^n \langle \gamma_N^i U , \gamma_D^i  D_k \rangle_{\partial \Omega_i^r } .
$$
From our previous computation and the fact that $D_k \in \mathcal{C}^2_0(\IR^2\setminus \overline{\Gamma})$, we have that $\sum_{i=1}^n \langle \gamma_N^i U , \gamma_D^i  D_k \rangle_{\partial \Omega_i^r }=0$. Furthermore, we have that
$$
\lim_{k \rightarrow \infty} \sum_{i=1}^n \langle \gamma_N^i U , \gamma_D^i (V- W - D_k) \rangle_{\partial \Omega_i^r } = \sum_{i=1}^n \langle \gamma_N^i U , \gamma_D^i (V- W - D) \rangle_{\partial \Omega_i^r }.
$$
An application of Green's formula yields
\begin{align*}
&\sum_{i=1}^n \langle \gamma_N^i U , \gamma_D^i (V- W - D) \rangle_{\partial \Omega_i^r } \\
&= \sum_{i=1}^n \int\limits_{\Omega_i^r} \Delta U(\bx) (V(\bx)-W(\bx)- D(\bx)) - \nabla U(\bx) \cdot \nabla (V(\bx)-W(\bx)-D(\bx))\ \d \bx.
\end{align*}
Moreover, since $D = V-W$ in $\Omega$ since the supports of $D$, $V$ and $W$ are all contained in $B_0(r)$, we have
\begin{equation*}
\begin{split}
    &\sum_{i=1}^n \langle \gamma_N^i U , \gamma_D^i (V- W - D) \rangle_{\partial \Omega_i^r } \\
    &= 
\sum_{i=1}^n \int\limits_{\Omega_i^r \setminus \Omega} \Delta U(\bx) (V(\bx)-W(\bx)- D(\bx)) - \nabla U(\bx) \cdot \nabla (V(\bx)-W(\bx)-D(\bx))\ \d \bx \\
&= \int\limits_{\IR^2 \setminus \Omega} \Delta U(\bx) (V(\bx)-W(\bx)- D(\bx)) - \nabla U(\bx) \cdot \nabla (V(\bx)-W(\bx)-D(\bx))\ \d \bx \\
&= \langle \gamma_{N,\partial \Omega} U, \gamma_{D,\partial \Omega} (V-W- D) \rangle _{\partial \Omega},
\end{split}
\end{equation*}
where $\gamma_{N,\partial \Omega}, \gamma_{D,\partial \Omega}$ are the (exterior) Neumann and Dirichlet traces on $\partial \Omega$. Then, since $D=V-W$ in $H^1(\Omega)$ and the fact $D$, $V$ and $W$ are continuous in $\IR^2$, we have that
$$
\langle \gamma_{N,\partial \Omega} U, \gamma_{D,\partial \Omega} (V-W- D) \rangle _{\partial \Omega} = 0,
$$
which implies that
$$
\sum_{i=1}^n \langle \gamma_N^i U , \gamma_D^i (V- W) \rangle_{\partial \Omega_i^r } = 0,
$$
and we have shown that $\lim_{r \rightarrow \infty} \sum_{i=1}^n \langle \gamma_N^i U, \gamma_D^i V\rangle_{\partial \Omega_i^r}$ is independent of the chosen extension of $v\in\C^\infty(\Gamma)$.

Finally take $v \in \C^\infty(\Gamma)$, an arbitrary extension $V \in \C_0^\infty(\IR^2)$, some fixed $r_0 \in \IR_\Gamma$ and consider a smooth window function $\chi$ such that $\chi(\bx) =1 $ in a non-specified open neighborhood of $\Gamma$ and $\supp(\chi) \subset B_{\bm{0}}(r_0)$.
Then, we have that 
$$
\lim_{r \rightarrow \infty} \sum_{i=1}^n \langle \gamma_N^i U , \gamma_D^i V \rangle_{\partial \Omega_i^r } = \lim_{r \rightarrow \infty} \sum_{i=1}^n \langle \gamma_N^i U , \gamma_D^i (\chi V) \rangle_{\partial \Omega_i^r } = \sum_{i=1}^n \langle \gamma_N^i U , \gamma_D^i (\chi V) \rangle_{\partial \Omega_i^{r_0} }.
$$
Another application of Green's formula yields
$$
\sum_{i=1}^n \langle \gamma_N^i U , \gamma_D^i (\chi V) \rangle_{\partial \Omega_i^{r_0}}
 = \sum_{i=1}^n \int\limits_{\Omega_i^{r_0}} \Delta U(\bx) \chi(\bx) V(\bx) + \chi(\bx) \nabla U(\bx) \cdot \nabla V(\bx)+ V(\bx)\nabla \chi(\bx) \cdot \nabla U(\bx)\ \d\bx, 
$$
which immediately implies that 
\begin{align*}
    \left\lvert \lim_{r \rightarrow \infty} \sum_{i=1}^n \langle \gamma_N^i U , \gamma_D^i V \rangle_{\partial \Omega_i^r } \right\rvert
 &\lesssim \sum_{i=1}^n \left(\Vert \Delta U \Vert_{L^2(\Omega_i^{r_0}) }+ |U|_{H^1(\Omega_i^{r_0})}|\chi|_{H^1(\Omega_i^{r_0})}\right)\Vert V\Vert_{H^1(\Omega_i^{r_0})} \\
 &\lesssim C_{r_0} \Vert V\Vert_{H^1(\IR^2)}
 \sum_{i=1}^n \left(\Vert \Delta U \Vert_{L^2(\Omega_i^{r_0}) }+ |U|_{H^1(\Omega_i^{r_0})}\right).
\end{align*}
Since the left-hand side of the last equation is independent of the chosen extension,
we may take the infimum over all extensions of $v$ and the result follows from \Cref{lem:equiv}
and the density of $\C_0^\infty(\IR^2)$ in $H^1(\IR^2)$, see \cite[Cap.~3]{mclean2000strongly}.
\end{proof}

\Cref{prop:boundNeumanJump} motivates the following operator as an extension of the jump of the normal derivative across multi-arc geometries.

\begin{definition}[Neumann Jump]
\label{def:NeuJump}
For all $U \in H^1_{\text{\text{loc}}}(\IR^2 \setminus \overline \Gamma; \Delta)$ 
its \emph{Neumann jump} is the functional $[\gamma_N U] \in (H^{\half}(\Gamma))^\star$ defined as the unique extension to $H^{\half}(\Gamma)$ of
\begin{align}
\label{eq:jumpNeuDef}
  \langle [\gamma_N U],v\rangle_\Gamma \coloneqq -\lim_{r\rightarrow \infty} \sum_{i=1}^n \langle \gamma_N^i U, \gamma_D^i V \rangle_{ \partial \Omega^r_i}\quad\forall v\in\C^\infty(\Gamma),
\end{align}
where $V \in \C_0^\infty(\IR^2)$ is an arbitrary extension of $v\in\C^\infty(\Gamma)$.
\end{definition}
Thanks to \Cref{prop:boundNeumanJump}, it is clear that the Neumann jump of $U$ is well defined.
The following result gives us a more flexible representation of the Neumann jump.
\begin{proposition}
\label{prop:NeuJumpExt}
Let $U \in H^1_{\text{\text{loc}}}(\IR^2 \setminus \overline \Gamma; \Delta)$ and, additionally, assume that:
\begin{enumerate}
\item There exists some $r_0>0$, such that $U \in H^2_{\text{loc}}({\overline{B_{\bm{0}}(r_0)}}^c)$. 
\item $U(\bx) = \cO(\Vert\bx\Vert^{-1})$ and $\nabla U(\bx) = \cO(\Vert\bx\Vert^{-2})$ as $\Vert\bx\Vert\rightarrow \infty$. 
\end{enumerate}
Then, for any $v \in H^{\half}(\Gamma)$ that has an extension $V$ in $H^1_{\text{loc}}(\IR^2)$ satisfying $V(\bx) = \cO(\Vert\bx\Vert^{-1})$ and $\nabla V(\bx) = \cO(\Vert\bx\Vert^{-2})$ as $\Vert\bx\Vert \rightarrow \infty$, we have
$$\langle [\gamma_N U], v \rangle_\Gamma = -\lim_{r \rightarrow \infty} \sum_{i=1}^n \langle \gamma_N^i U, \gamma_D^i V \rangle_{\partial \Omega_i^r}.$$
\end{proposition}
\begin{proof}
Let $r_1,r_2 \in \IR_{\Gamma}$, with $r_2 > r_1>r_0$, and let $\chi\in\C^\infty_0(\IR^2)$ be a smooth window function with compact support in $B_{\bm{0}}(r_2)$, such that $\chi(\bx) = 1$
for all $\bx \in B_{\bm{0}}(r_1)$. Take $v \in H^{\frac{1}{2}}(\Gamma)$ and an extension
$V \in H^1_{\text{loc}}(\IR^2)$ satisfying the assumptions in the statement.
By the density of $\C^\infty_0(\IR^2)$ in $H^1(B_{\bm{0}}(r_2))$, there exists a sequence $\{V_k\}_{k \in \IN} \subset \C_0^\infty(\IR^2)$
which converges to $V$ in $H^1(B_{\bm{0}}(r_2))$, for which it holds that
\begin{align*}
\langle 
[\gamma_NU],v\rangle_\Gamma
= \lim_{k \rightarrow \infty}\langle [\gamma_NU],\gamma_D^i \chi V_k\rangle_\Gamma
= -\lim_{k \rightarrow \infty}\lim_{r \rightarrow \infty}
\sum_{i=1}^n \langle\gamma_N^i U,\gamma_D^i \chi V_k\rangle_{\partial\Omega_i^r}.
\end{align*}

Furthermore, we have that
\begin{align}\nonumber
-\lim_{k \rightarrow \infty}\lim_{r \rightarrow \infty} &
\sum_{i=1}^n \langle\gamma_N^i U,\gamma_D^i \chi V_k\rangle_{\partial\Omega_i^r} +  \lim_{r \rightarrow \infty} \sum_{i=1}^n \langle \gamma_N^iU, \gamma_D^i V\rangle_{\partial \Omega_i^r} \\
\nonumber
&= -\lim_{k \rightarrow \infty} \lim_{r \rightarrow \infty}\sum_{i=1}^n \langle 
\gamma_N^i U,\gamma_D^i \chi (V_k-V)\rangle_{\partial \Omega_i^{r}} + \sum_{i=1}^n \langle \gamma_N^i U, \gamma_D^i (1-\chi)V\rangle_{\partial \Omega_i^r}\\
\label{eq:proplimklimr}
&= -\lim_{k \rightarrow \infty} \sum_{i=1}^n \langle 
\gamma_N^i U,\gamma_D^i \chi (V_k-V)\rangle_{\partial \Omega_i^{r_2}} + \lim_{r \rightarrow \infty} \sum_{i=1}^n \langle \gamma_N^i U, \gamma_D^i (1-\chi)V\rangle_{\partial \Omega_i^r}.
\end{align}
For the first term in \cref{eq:proplimklimr}, we can follow the same strategy as in proof of \Cref{prop:boundNeumanJump} (namely, Green's formula) to obtain:
 $$
 \left\lvert\sum_{i=1}^n \langle 
\gamma_N^i U,\gamma_D^i \chi (V_k-V)\rangle_{\partial \Omega_i^{r_2}}\right\rvert\lesssim C_{r_2} \left(\sum_{i=1}^n \Vert\Delta U\Vert_{L^2(\Omega_i^{r_2})} + \Vert \nabla U \Vert_{L^2(\Omega_i^{r_2})}\right) \Vert V_k- V\Vert_{H^{1}(B_{\bm{0}}(r_2))}.
 $$
For the second term in \cref{eq:proplimklimr}, assuming that $r>r_1$ and noting that
$U\in H^2(B_0(r)\setminus \overline{B_0(r_1)})$ and that $1-\chi(\bx)=0$ in $\overline{B_{\bm{0}}(r_1)}$, it holds that
\begin{align*}
\sum_{i=1}^n \langle 
\gamma_N^i U,\gamma_D^i (1-\chi) V\rangle_{\partial \Omega_i^r}
&=
\sum_{i=1}^n 
\int\limits_{\Omega_i^r} \Delta U(\bx)(1-\chi(\bx)) V(\bx) + \nabla U(\bx)\nabla\left((1-\chi(\bx)) V(\bx)\right)\ \d\bx
\\&=
\sum_{i=1}^n 
\int\limits_{\Omega_i^r\setminus B_0(r_1)} \Delta U(\bx)(1-\chi(\bx)) V(\bx) + \nabla U(\bx)\nabla\left((1-\chi(\bx)) V(\bx)\right)\ \d\bx
\\&=
\int\limits_{B_0(r)\setminus B_0(r_1)} \Delta U(\bx)(1-\chi(\bx)) V(\bx) + \nabla U(\bx)\nabla\left((1-\chi(\bx)) V(\bx)\right)\ \d\bx
\\&=
\int\limits_{\partial (B_0(r) \setminus  B_{\bm{0}}(r_1)) } \partial_{\bm{n}} U(\bx)(1-\chi(\bx)) V(\bx)\ \d\bx 
\\&=
\int\limits_{\partial B_0(r)} \partial_{\bm{n}} U(\bx)(1-\chi(\bx)) V(\bx)\ \d\bx,
 \end{align*}
 where $\partial_{\bm{n}} U$ denotes the normal derivative of $U$. Furthermore,
$$ \left\lvert \int\limits_{\partial B_{\bm{0}}(r)} \partial_{\bm{n}} U(\bx) (1-\chi(\bx)) V(\bx)\ \d\bx \right\rvert \lesssim \int\limits_{\partial B_{\bm{0}}(r)} \frac{1}{\Vert \bx\Vert ^3}\ \d\bx 
=2 \pi\frac{1}{r^2}.
$$
Therefore, both of the terms in \cref{eq:proplimklimr} converge to $0$ as $k$ and $r$ grow towards infinity, yielding the result of the proposition.

\begin{remark}
    The previous result is sub-optimal, in the sense that we only need that $$\lim_{\Vert\bx\Vert \rightarrow \infty}  \Vert x\Vert \vert U(\bx)\vert \vert V(\bx)\vert = 0,$$ for the result to holds. However, we will only use the above characterization of the Neuman jump for the case where $U,V$ are integral potentials, which have the supposed decay properties .  
\end{remark}

\end{proof}

\subsection{Integral Formulation for Dirichlet Problem}
\label{sec:DirProb}
In this section, we consider the Laplace equation, as shown in \cref{eq:lap},
subject to the inhomogeneous Dirichlet boundary condition
\[
\gamma_D U = f\quad \text{on } \Gamma,
\]
where $f$ is some prescribed boundary datum (the precise regularity assumptions on $f$ will be specified below). Our goal is to reformulate this boundary value problem as a boundary integral equation on the multi-arc $\Gamma$.

\paragraph{Single-Layer Potential.} Let $G(\bx)$ denote the fundamental solution of the Laplacian in $\IR^2$, i.e.,
\begin{align*}
G(\bx) \coloneqq -\frac{1}{2\pi}\log \Vert\bx\Vert.
\end{align*}
For a given density $\lambda \in L^2(\Gamma)$, we define the \emph{single-layer potential} as
\begin{align*}
\cS\lambda(\bx) \coloneqq \int\limits_{\Gamma} G(\bx-\by)\,\lambda(\by)\ \d\by, \quad \bx \in \IR^2 \setminus \overline{\Gamma}.
\end{align*}
The single-layer on $\Gamma$ is related to the classical single-layer potential on open arcs as follows:
\begin{equation}
\label{eq:SLrep}
\cS\lambda(\bx)
= \half\sum_{i=1}^n \int\limits_{\widehat{\Gamma}_i}
G(\bx-\by)\,\lambda_i(\by)\,d\by
= \half\sum_{i=1}^n \cS_{\widehat{\Gamma}_i}\lambda(\bx),
\end{equation}
where, for each $i\in\{1,\hdots,n\}$, $\cS_{\widehat{\Gamma}_i}$ denotes the single-layer potential associated with the open arc $\widehat{\Gamma}_i$.
This last equation motivates the following definition for the single-layer potential acting on $\lambda\in H^{-\frac{1}{2}}(\Gamma)$.

\begin{definition}[Single-Layer Potential]
\label{defi:SL}
For $\lambda\in \widetilde{H}^{-\frac{1}{2}}(\Gamma)$, the single-layer potential acting on the density $\lambda$
is defined through the single-layer potentials associated with the open arcs $\{\widehat\Gamma_i\}_{i=1}^n$
(as in \Cref{def:multiarc}) as
\begin{equation*}
\cS\lambda
:= \half\sum_{i=1}^n \cS_{\widehat{\Gamma}_i}\lambda.
\end{equation*}
\end{definition}

\begin{proposition}
\label{prop:SL}
  The single-layer potential $\cS$ on $\Gamma$ satisfies the following properties:
\begin{enumerate}
  \item $\displaystyle
  \cS = \half\sum_{i=1}^n \mathcal{N}\circ{\gamma_D^{i,\star}}$,
  where, for each $i\in\{1,\hdots,n\}$, ${\gamma_D^{i,\star}}$ denotes the adjoint of the Dirichlet trace on $\widehat{\Gamma}_i$, and $\mathcal{N}$ is the Newton potential on $\IR^2$.
  \item $\Delta (\cS\lambda) = 0$ on $\IR^2\setminus\overline{\Gamma}$.
  \item For any $\chi\in C_0^\infty(\IR^2)$,
  \[
  \chi\cS : \widetilde{H}^{-\half}(\Gamma)\to H^1(\IR^2)
  \]
  is a bounded and linear operator.
  \item $\cS\lambda\in H^2(\omega)$ for any open and bounded set
  $\omega\subset\IR^2\setminus\overline{\Gamma}$.
\end{enumerate}
\end{proposition}
\begin{proof}
For closed boundaries, these properties are already known to hold (see, for example,
\cite[Chap.~6]{mclean2000strongly}). Their extension to open arcs is discussed in
\cite[Sec.~3.5.3]{Sauter:2011}, from where the multi-arc case follows directly from \Cref{defi:SL}.
\end{proof}

An immediate consequence of \Cref{prop:SL} is that the single-layer potential has \jp{zero} Dirichlet jump across $\Gamma$. 
We continue by analyzing, in the following Lemma, the Neumann jump of the single-layer potential.
We note that its proof closely follows the arguments of the proof of \cite[Thm.~3.3.1]{Sauter:2011}.
\begin{lemma}
\label{lem:Njump}
  If $\lambda \in \widetilde{H}^{-\half}(\Gamma)$, then 
  $$
  [\gamma_N \cS\lambda] = -\lambda \quad\text{in}\quad\widetilde{H}^{-\half}(\Gamma).
  $$
\end{lemma}
\begin{proof}
  Let $v \in \C^\infty(\Gamma)$, $V \in C_0^\infty(\IR^2)$ be an extension of $v$ and set $U=\cS\lambda$ for some $\lambda \in \widetilde{H}^{-\half}(\Gamma)$. Then,
  it holds that
  \begin{align*}
  \langle \Delta U , V \rangle_{\IR^2} = \langle U , \Delta V \rangle_{\IR^2}
  & = \lim_{r \rightarrow \infty} \sum_{i=1}^n \int\limits_{\Omega_i^r} \Delta V(\bx) U(\bx)\ \d\bx\\
  & = \lim_{r \rightarrow \infty} \sum_{i=1}^n -\int\limits_{\Omega_i^r} \nabla V(\bx) \nabla U(\bx)\ \d\bx + \langle \gamma_N^i V, \gamma_D^i U \rangle_{\partial \Omega_i^r}\\
  & = \lim_{r \rightarrow \infty} \sum_{i=1}^n \int\limits_{\Omega_i^r} \Delta U(\bx)  V(\bx)\ \d\bx - \langle \gamma_N^i U , \gamma_D^i V \rangle_{\partial \Omega_i^r}  + \langle \gamma_N^i V, \gamma_D^i U \rangle_{\partial \Omega_i^r}\\
  & = \lim_{r \rightarrow \infty} \sum_{i=1}^n \langle \gamma_N^i V, \gamma_D^i U \rangle_{\partial \Omega_i^r} - \langle \gamma_N^i U , \gamma_D^i V \rangle_{\partial \Omega_i^r}, 
  \end{align*}
  where, in the last step, we have used that $\Delta U=0$ on $\IR^2\setminus\overline\Gamma$ (see \Cref{prop:SL}).
  The term
    $$
    \lim_{r \rightarrow \infty} \sum_{i=1}^n \langle \gamma_N^i V, \gamma_D^i U \rangle_{\partial \Omega_i^r}
    $$
  corresponds to the Neumann jump of $V$ acting on $U$ (\emph{c.f.}~\Cref{prop:NeuJumpExt,prop:SL}) and is, therefore, $0$ due to the smoothness properties of $V$ and $U$.
  On the other hand, since the Newton potential is the inverse of the $-\Delta$ operator (see \cite[Theorem 3.1.4]{Sauter:2011}), we have that
$$
\langle \Delta U, V \rangle_{\IR^2} = \half\sum_{i=1}^n \langle \Delta \mathcal{N} (\gamma_D^{i,\star}(\widetilde{\lambda})), V\rangle_{\IR^2} = -\half \sum_{i=1}^n \langle \lambda , v\rangle_{\widehat{\Gamma}_i} = -\langle \lambda , v \rangle_{\Gamma},
$$
and we may conclude that 
$$
-\lim_{r \rightarrow \infty} \sum_{i=1}^n \ \langle \gamma_N^i U , \gamma_D^i V \rangle_{\partial \Omega_i^r} =-\langle \lambda , v \rangle_{\Gamma},
 $$
 from where the proof is completed by recalling the definition of the Neumann jump in \Cref{def:NeuJump}.
\end{proof}
\paragraph{Integral Equation and Weakly Singular Operator.} Since $\cS\lambda$ is harmonic in $\IR^2\setminus\overline{\Gamma}$, it is natural to seek solutions to the Dirichlet problem of the form
\begin{align*}
  U = \cS \lambda,
\end{align*}
where $\lambda\in \widetilde H^{-\frac{1}{2}}(\Gamma)$ is an unknown density. Imposing the Dirichlet boundary condition to our representation leads to the boundary integral equation for the Dirichlet problem:
\begin{equation}
\label{eq:DirBIE}
  \gamma_D \cS \lambda = f.
\end{equation}
We therefore introduce the \emph{weakly singular} operator as
$$
\mathcal{V} \lambda  \coloneq \gamma_D\cS \lambda.
$$
\begin{proposition}
The weakly singular operator
$$
\mathcal{V}: \widetilde{H}^{-\half}(\Gamma) \rightarrow H^{\half}(\Gamma),
$$
is linear and bounded on $\widetilde{H}^{-\half}(\Gamma)$. 
\end{proposition}
\begin{proof}
This follows immediately from the mapping properties of the single-layer potential (in \Cref{prop:SL}) and the Dirichlet trace.
\end{proof}

To establish the ellipticity of the weakly singular operator, we require the following property which states the decay of the single-layer potential.

\begin{lemma}
\label{lemma:lapdecay}
  Let $\lambda \in \widetilde{H}^{-\half}(\Gamma)$ and set $U = \cS\lambda $. Then
$$
\vert U(\bx)\vert = \sum_{i=1}^n \langle \lambda, 1\rangle_{\widehat{\Gamma}_i} \log\Vert x\Vert+ \cO(\Vert x\Vert^{-1}), \quad \text{as } \bx \rightarrow \infty.
$$
  If, in addition, it holds that $\langle \lambda, 1\rangle_{{\Gamma}} = 0$, then we have that
  \begin{align}
    \vert U(\bx)\vert \lesssim \Vert\lambda\Vert_{\widetilde{H}^{-\half}(\Gamma)}\Vert\bx\Vert^{-1}, \quad \text{and} \quad \Vert\nabla U(\bx)\Vert \lesssim \Vert\lambda\Vert_{\widetilde{H}^{-\half}(\Gamma)} \Vert\bx\Vert^{-2}, \quad \text{as } \Vert\bx\Vert \rightarrow \infty. 
  \end{align}
\end{lemma}
\begin{proof}
    The result follows analogously to that for closed boundaries, which may be found in \cite[Lemma 6.21]{STEI07}.
\end{proof}

With these properties at hand, we find ourselves in the position to prove the ellipticity of the weakly singular operator.

\begin{proposition}
\label{prop:elipV}
  For any $\lambda \in \widetilde{H}^{-\half}(\Gamma)$ satisfying $\langle \lambda, 1\rangle _{\Gamma} = 0$, it holds that
  $$
  \Vert\lambda\Vert^2_{\widetilde{H}^{-\half}(\Gamma)} \lesssim \langle \mathcal{V} \lambda, \lambda \rangle_\Gamma.
  $$
\end{proposition}



\begin{proof}
  Again. we set $U = \cS\lambda$ and, by \Cref{lem:Njump}, we have that 
  $
  [\gamma_N U ] = -\lambda 
  $. A combination of \Cref{lemma:lapdecay} and \Cref{prop:NeuJumpExt} yields 
  $$
  \langle \lambda , \mathcal{V}\lambda\rangle_{\Gamma} = \langle [\gamma_N U] , \gamma_D U\rangle_{\Gamma}= -
  \lim_{r \rightarrow \infty} \sum_{i= 1}^n \langle \gamma_N^i U , \gamma_D^i U\rangle_{\partial \Omega_i^r}.
  $$
Using Green's formula together with the fact that $U=\cS\lambda$ solves the Laplace equation outside of $\Gamma$, we obtain that
$$
  \langle \lambda , \mathcal{V}\lambda\rangle_{\Gamma} = \lim_{r \rightarrow \infty} \sum_{i=1}^n \vert U\vert_{H^1(\Omega_i^r)}^2=\lim_{r \rightarrow \infty}\vert U\vert_{H^1(B_0(r))}^2.
$$
Notice that, thanks to \Cref{lemma:lapdecay}, the right-hand side of the last equation is finite. Furthermore, by incorporating this condition, and the fact that $U$ is harmonic in $\IR^2\setminus \overline \Gamma$, we can improve the bound of \Cref{prop:boundNeumanJump} so we obtain  
$$\Vert \lambda\Vert^2_{\widetilde{H}^{-\frac{1}{2}}(\Gamma)} =\Vert [\gamma_N U]\Vert^2_{(H^{\frac{1}{2}}(\Gamma))^\star} \lesssim \lim_{r\rightarrow \infty} \sum_{i=1}^n \vert U\vert_{H^1(B_0(r))}^2=\langle \lambda , \mathcal{V}\lambda\rangle_{\Gamma},$$
and we have proved the ellipticity of the weakly singular operator. 
\end{proof}
The result in \Cref{prop:elipV} implies that the weakly singular integral equation must be posed in the space
$$
\widetilde{H}^{-\half}_*(\Gamma): = \{ \lambda \in \widetilde{H}^{-\half}(\Gamma): \langle \lambda, 1\rangle_{\jp{{\Gamma}}} = 0  \}.
$$
Doing so allows us to prove the existence of a unique solution through the Lax-Milgram Theorem. 

\begin{theorem}
\label{thrm:uniquesol}
  Let $f$ be a given element of the dual space of $\widetilde{H}_*^{-\half}(\Gamma)$. Then, there exists a unique solution $\lambda \in \widetilde{H}^{-\half}_*(\Gamma)$ such that
  \begin{align}
  \label{eq:Veq}
       \langle v , \mathcal{V} \lambda \rangle_\Gamma = f(v) \quad \forall v \in \widetilde{H}_*^{-\half}(\Gamma). 
  \end{align}
\end{theorem}
\begin{proof}
Let $v \in \widetilde{H}^{-\half}(\Gamma)$. Then, we have that
$$
\langle v , \mathcal{V} v\rangle_\Gamma = \half \sum_{i=1}^n \langle v , \gamma_D \cS v \rangle_{\widehat{\Gamma}_i} \lesssim \sum_{i=1}^n \Vert v \Vert_{\widetilde{H}^{-\half}(\widehat{\Gamma}_i)} \Vert \gamma_D \cS v \Vert_{H^{\half}(\widehat{\Gamma}_i)} \lesssim \sum_{i=1}^n \Vert v \Vert^2 _{\widetilde{H}^{-\half}(\widehat{\Gamma}_i)} = \Vert v\Vert_{\widetilde{H}^{-\half}(\Gamma)}^2,
$$
and the bilinear form is continuous and (by \Cref{prop:elipV}) elliptic. The result is then a direct consequence of the Lax-Milgram theorem.
\end{proof}

\begin{remark} Although the analysis above is carried out in $\widetilde{H}^{-\half}(\Gamma)$, all the results extend to $(H^{\half}(\Gamma))^\star$, which would yield equivalent results to those in \cite{Claeys2013}. However, $\widetilde{H}^{-\half}(\Gamma)$ is a more natural choice, as it leads to the classical discrete approximation spaces. \end{remark}

\subsection{Integral Formulation for the Neumann Problem}
\label{sec:NeuProb}

In this section, we derive a boundary integral formulation for the Laplace equation with Neumann boundary conditions on $\Gamma$. To simplify the presentation, we restrict ourselves to the case in which $\Gamma$ is a \emph{triple junction}, that is, it consists of three open arcs meeting at a single, common point (see \Cref{fig:NeumanGeo}). Despite its simplicity, this geometry captures the main analytical difficulties arising in multi-arc configurations, in particular, the presence of branch points and the interaction between traces defined on different open arcs. At the end of the present section, we briefly discuss how the analysis therein may be extended to more general multi-arc geometries, with only minor modifications. 

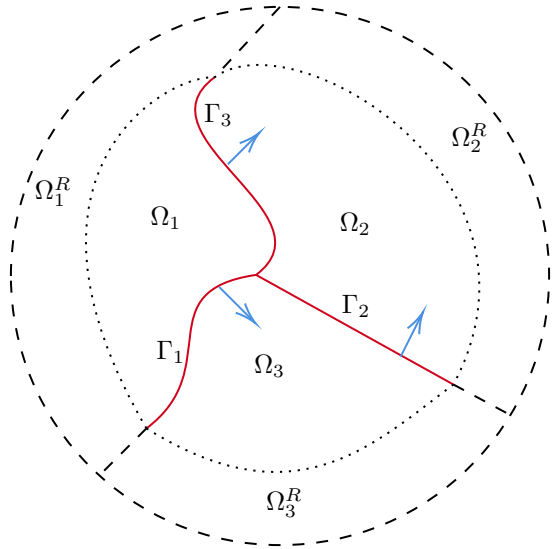
\begin{figure}[ht]
  \centering
\tikzset{every picture/.style={line width=0.75pt}} 

\begin{tikzpicture}[x=0.75pt,y=0.75pt,yscale=-1,xscale=1]

\draw [color={rgb, 255:red, 208; green, 2; blue, 27 } ,draw opacity=1 ]  (260.5,151.95) .. controls (300.5,121.95) and (199.5,82.95) .. (239.5,52.95) ;
\draw [color={rgb, 255:red, 208; green, 2; blue, 27 } ,draw opacity=1 ]  (260.5,151.95) -- (359.5,206.95) ;
\draw [color={rgb, 255:red, 208; green, 2; blue, 27 } ,draw opacity=1 ]  (205.5,228.95) .. controls (245.5,198.95) and (205.5,159.95) .. (260.5,151.95) ;
\draw [dash pattern={on 0.84pt off 2.51pt}] (239.5,52.95) .. controls (279.5,22.95) and (414.5,103.95) .. (359.5,206.95) ;
\draw [dash pattern={on 0.84pt off 2.51pt}] (205.5,228.95) .. controls (129.5,98.95) and (217.5,50.95) .. (239.5,52.95) ;
\draw [dash pattern={on 0.84pt off 2.51pt}] (205.5,228.95) .. controls (280.5,275.95) and (319.5,236.95) .. (359.5,206.95) ;
\draw [dash pattern={on 4.5pt off 4.5pt}] (137.5,152.5) .. controls (137.5,77.94) and (197.94,17.5) .. (272.5,17.5) .. controls (347.06,17.5) and (407.5,77.94) .. (407.5,152.5) .. controls (407.5,227.06) and (347.06,287.5) .. (272.5,287.5) .. controls (197.94,287.5) and (137.5,227.06) .. (137.5,152.5) -- cycle ;
\draw [color={rgb, 255:red, 0; green, 0; blue, 0 } ,draw opacity=1 ] [dash pattern={on 4.5pt off 4.5pt}] (205.5,228.95) -- (181.5,252.95) ;
\draw [color={rgb, 255:red, 0; green, 0; blue, 0 } ,draw opacity=1 ] [dash pattern={on 4.5pt off 4.5pt}] (359.5,206.95) -- (387.5,221.95) ;
\draw [dash pattern={on 4.5pt off 4.5pt}] (272.5,17.5) -- (239.5,52.95) ;
\draw [color={rgb, 255:red, 74; green, 144; blue, 226 } ,draw opacity=1 ]  (242,158) -- (258.92,174.92) ;
\draw [shift={(260.33,176.33)}, rotate = 225] [color={rgb, 255:red, 74; green, 144; blue, 226 } ,draw opacity=1 ][line width=0.75]  (10.93,-3.29) .. controls (6.95,-1.4) and (3.31,-0.3) .. (0,0) .. controls (3.31,0.3) and (6.95,1.4) .. (10.93,3.29)  ;
\draw [color={rgb, 255:red, 74; green, 144; blue, 226 } ,draw opacity=1 ]  (333,193) -- (343.43,172.25) ;
\draw [shift={(344.33,170.46)}, rotate = 116.69] [color={rgb, 255:red, 74; green, 144; blue, 226 } ,draw opacity=1 ][line width=0.75]  (10.93,-3.29) .. controls (6.95,-1.4) and (3.31,-0.3) .. (0,0) .. controls (3.31,0.3) and (6.95,1.4) .. (10.93,3.29)  ;
\draw [color={rgb, 255:red, 74; green, 144; blue, 226 } ,draw opacity=1 ]  (246.33,96.46) -- (260.92,81.87) ;
\draw [shift={(262.33,80.46)}, rotate = 135] [color={rgb, 255:red, 74; green, 144; blue, 226 } ,draw opacity=1 ][line width=0.75]  (10.93,-3.29) .. controls (6.95,-1.4) and (3.31,-0.3) .. (0,0) .. controls (3.31,0.3) and (6.95,1.4) .. (10.93,3.29)  ;

\draw (206,116.4) node [anchor=north west][inner sep=0.75pt]  {$\Omega _{1}$};
\draw (301,119.4) node [anchor=north west][inner sep=0.75pt]  {$\Omega _{2}$};
\draw (258,190.4) node [anchor=north west][inner sep=0.75pt]  {$\Omega _{3}$};
\draw (209,183.4) node [anchor=north west][inner sep=0.75pt]  {$\Gamma _{1}$};
\draw (302,160.4) node [anchor=north west][inner sep=0.75pt]  {$\Gamma _{2}$};
\draw (233,64.4) node [anchor=north west][inner sep=0.75pt]  {$\Gamma _{3}$};
\draw (148,101.4) node [anchor=north west][inner sep=0.75pt]  {$\Omega _{1}^{R}$};
\draw (357,74.4) node [anchor=north west][inner sep=0.75pt]  {$\Omega _{2}^{R}$};
\draw (264,257.4) node [anchor=north west][inner sep=0.75pt]  {$\Omega _{3}^{R}$};

\end{tikzpicture}
\caption{Graphical representation of a multi-screen $\Gamma$ constructed from $\Gamma_1, \Gamma_2, \Gamma_3$, with a fixed orientation of the normal vectors.}
  \label{fig:NeumanGeo}
\end{figure}

We begin our presentation by giving a notion for the Neumann trace on the multi-arc $\Gamma$.
\begin{definition}
\label{def:neuTrace}
For $U \in H^1_{\text{\text{loc}}}(\IR^2 \setminus \overline \Gamma; \Delta)$, the \emph{Neumann trace} of $U$ on $\Gamma$ is denoted by $\gamma_N U$, and is defined as an element of  ${H}^{-\frac{1}{2}}_\times(\Gamma)$ by the duality relation
$$
\langle \gamma_N U , \mu \rangle_{H^{-\half}_\times(\Gamma),\widetilde{H}^{\half}_\times(\Gamma)} \coloneq \langle \gamma_N^1 U, \mu_1\rangle_{\widehat{\Gamma}_1} - \langle \gamma_N^2 U, \mu_2 \rangle_{\widehat{\Gamma}_2}, 
$$
for all $\mu \in \widetilde {H}^{\half}_\times(\Gamma)$.
\end{definition}
\begin{remark}
Notice that this is not a generalization of the classical normal derivative, but instead for a smooth function $U$ its Neumann trace can be formally identified with the following function
$$
\gamma_N U (\bx) = \begin{cases}
  2 \partial_{\bm{n}}U(\bx), \quad \bx \in \Gamma_3 \\
  \partial_{\bm{n}} U (\bx), \quad \bx \in \Gamma_1 \\
  \partial_{\bm{n}} U(\bx), \quad \bx \in \Gamma_2 
\end{cases},
$$
where $\partial_{\bm{n}} U$ denotes the classical normal derivative.
 At the present time, this is the only operator that determined by the Nuemann traces on the $\widehat{\Gamma}_1, \widehat{\Gamma}_2, \widehat{\Gamma}_3$ (and hence it can be studied on the Sobolev scale previously introduced), that is somehow connected to the standard Neumann trace.  Notice that,  in contrast to closed boundaries, Green's formula can only generalize the notion of the jump of the normal derivative, but not the normal derivative itself.
An alternative approach, based on combining the individual traces $\gamma_N^i U$ with suitable $L^\infty(\Gamma)$ weights correcting the orientation of the normals, encounters additional difficulties since the product of an $H^{-\half}(\Gamma)$ distribution with an $L^\infty(\Gamma)$ function is not well defined in general.
\end{remark}

In what follows, we assume that the Neumann boundary condition is prescribed as
\begin{align*}
  \gamma_N U = f\quad \text{in }H^{-\frac{1}{2}}_\times(\Gamma)
\end{align*}
for some given $f\in H^{-\frac{1}{2}}_\times(\Gamma)$. Moreover, since $\gamma_N U \in  {H}^{-\half}_\times(\Gamma)$, it induces a continuous linear functional on $\widetilde{H}^{\half}(\Gamma)$. Therefore, we identify $\gamma_N U$ with its representative in $(\widetilde{H}^{\half}(\Gamma))^\star$, characterized by
$$
\langle \gamma_N U, \mu \rangle_{\Gamma} \coloneq \langle \gamma_N^1 U, \mu\rangle_{\widehat{\Gamma}_1} - \langle \gamma_N^2 U, \mu \rangle_{\widehat{\Gamma}_2}, \quad \forall \mu \in \widetilde{H}^{\half}(\Gamma).
$$

\paragraph{Double-Layer Potential.} As in the Dirichlet case, we introduce an integral potential to represent solutions. 

\begin{definition}[Double-Layer Potential]
\label{def:doublelayer}
For $\mu \in \widetilde{H}^{\half}(\Gamma)$, the \emph{Double-Layer potential} acting on the density $\mu$ is defined through the double-layer potentials associated with the open arcs $\widehat{\Gamma}_1$ and $\widehat{\Gamma}_2$ as
$$
\cD \mu \coloneq \cD_{\widehat{\Gamma}_1} \mu - \cD_{\widehat{\Gamma}_2} \mu,
$$
where, for $i\in\{1,2\}$, $\mathcal{D}_{\widehat{\Gamma}_i}$ denotes the classical double-layer potential on the open arc $\widehat{\Gamma}_i$, with the normal vectors pointing to the exterior of $\Omega_i$.
\end{definition}
The mapping properties of $\mathcal{D}$ follow directly from those of the classical double-layer potential.
\begin{proposition}
\label{prop:DL}
The double-layer potential $\cD$ on $\Gamma$ satisfies the following properties:
  \begin{enumerate}
  \item The double-layer potential admits the following representation 
  $$
  \cD = \mathcal{N}\circ \gamma_N^{1,\star} -\mathcal{N}\circ \gamma_N^{2,\star}, 
  $$
  where, for $i\in\{1,2\}$, $\gamma_N^{i,\star}$ denotes the adjoint of the Neumann trace $\gamma_N^i$ restricted to $\widehat{\Gamma}_i$. 
  \item $\Delta \cD \mu = 0 \quad \text{on } \IR^2 \setminus \overline{\Gamma}$. 
  \item For any $\chi\in\C_0^\infty(\IR^2)$,
  $$ \chi \cD : \widetilde{H}^{\half}(\Gamma) \rightarrow H^1(\IR^2 \setminus \overline{\Gamma})$$
   is a bounded and linear operator. 
   \item $\mathcal{D}\mu\in H^2(\omega)$ for any bounded open set
  $\omega\subset\IR^2\setminus\overline{\Gamma}$.
\end{enumerate}
\end{proposition}
\begin{proof}
We again refer to \cite[Chap.~6]{mclean2000strongly} for the corresponding results for the case of closed boundaries. The extension to open arcs is also presented in \cite[Sec.~3.5.3]{Sauter:2011}. We then obtain the results for multi-arcs directly from the definition of the double-layer potential. 
\end{proof}

\paragraph{Integral Equation and Hypersingular Operator.} In an analogous manner as for the Dirichlet problem, we seek a solution for the Neumann problem of the form
$$
U = -\mathcal{D} \mu,
$$
where $\mu\in\widetilde{H}^{\half}(\Gamma)$ is an unknown density. Taking the Neumann trace of the previous equation leads to the following boundary integral equation,
$$
- \gamma_N \mathcal{D} \mu = -f,
$$
which motivates the definition of the \emph{hypersingular} operator as
$$
\mathcal{W} \mu \coloneq -\gamma_N \mathcal{D} \mu. 
$$
\begin{proposition}
The hypersingular operator
$$
\mathcal{W} : \widetilde{H}^{\half}(\Gamma) \rightarrow \widetilde{H}^{\half}(\Gamma)^\star,  
$$
is linear and bounded on $\widetilde{H}^{\half}(\Gamma)$. 
\end{proposition}
\begin{proof}
From \Cref{prop:DL} we have that $\mathcal{D} : \widetilde{H}^{\frac{1}{2}}(\Gamma) \rightarrow H^1(\IR^2 \setminus \overline{\Gamma}; \Delta)$ is a bounded operator. The result then follows directly from the boundedness of the Neumann trace acting in $H^1(\IR^2 \setminus \overline{\Gamma};\Delta)$.
 \end{proof}

\begin{lemma}
\label{lemma:DLdecay}
Let $\mu \in \widetilde{H}^{\half}(\Gamma)$ and set $U =\mathcal{D} \mu$. Then, it holds that
\begin{align}
\vert U(\bx)\vert \lesssim \Vert\mu\Vert_{\widetilde{H}^{\half}(\Gamma)}\Vert\bx\Vert^{-1} \quad \text{and} \quad \Vert\nabla U(\bx)\Vert \lesssim \Vert\mu\Vert_{\widetilde{H}^{\half}(\Gamma)} \Vert\bx\Vert^{-2}, \quad \text{as } \Vert\bx\Vert \rightarrow \infty. 
\end{align}
\end{lemma}
\begin{proof}
These properties follow directly from the integral representation of the double-layer potential (see \cite[Thm.~6.24]{STEI07}).
\end{proof}

\begin{proposition}
The hypersingular operator is injective and positive semi-definite, i.e.,
$$
\langle \mathcal{W} \mu , \mu \rangle_\Gamma \ge 0, \quad \forall \mu \in \widetilde{H}^{\half}(\Gamma).
$$
\end{proposition}
\begin{proof}
We begin by showing that the hypersingular operator is positive semi-definite. Let $\mu \in \widetilde{H}^{\half}(\Gamma)$ and set $U = -\mathcal{D} \mu$.
Then, for any $V \in \C_0^\infty(\IR^2)$, we have that
$$
\langle \Delta U, V\rangle_{\IR^2} = \langle -\Delta \mathcal{N}(\gamma_N^{1,\star}(\mu))-\Delta\mathcal{N}(\gamma_N^{2,\star}(\mu)) , V \rangle_{\IR^2} = \langle \gamma_N^1 V , \mu \rangle_{\widehat{\Gamma}_1} - \langle \gamma_N^2 V , \mu \rangle_{\widehat{\Gamma}_2} = \langle \gamma_N V , \mu \rangle_{\Gamma}.
$$
Furthermore, arguing as in the proof of \Cref{lem:Njump}, we have that
$$
\langle \Delta U , V\rangle_{\IR^2} = \lim_{r \rightarrow \infty} \sum_{i=1}^3 \langle \gamma_N^i V, \gamma_D^i U \rangle_{\partial \Omega_i^r} - \langle \gamma_N^i U, \gamma_D^i V \rangle_{\partial \Omega_i^r}, 
$$
from where we deduce that
\begin{align}
\nonumber
-\lim_{r \rightarrow \infty} \sum_{i=1}^3 \langle \gamma_N^i U, \gamma_D^i V \rangle_{\partial \Omega_i^r}
&=\langle \Delta U , V\rangle_{\IR^2} -\lim_{r \rightarrow \infty} \sum_{i=1}^3 \langle \gamma_N^i V, \gamma_D^i U \rangle_{\partial \Omega_i^r}\\
\label{eq:jumpsDL}
&=\langle \gamma_N V , \mu \rangle_{\Gamma} - \lim_{r \rightarrow \infty} \sum_{i=1}^3 \langle \gamma_N^i V, \gamma_D^i U \rangle_{\partial \Omega_i^r}.
\end{align}
Now, for some fixed $r_0 \in \IR_\Gamma$ and for $i\in\{1,2\}$, we introduce
$Z_i$ as the unique solution of the following problem:
\begin{align*}
  -\Delta Z_i &= 0 \quad \text{in } \Omega_i^{r_0}, \\
  \gamma_D^i Z_i &= \widetilde{\mu} \quad \text{on } \partial \Omega_i^{r_0},
\end{align*}
where $\widetilde{\mu}$ is the extension by $0$ of $\mu$ to each $\partial\Omega_i^{r_0}$. Then, we define $Z\in L^2(B_0(r_0))$ as 
$$
Z(\bx) := \begin{cases}
   Z_1(\bx), \quad &\text{for } \bx \in \Omega_1^{r_0} \\
   -Z_2(\bx), \quad &\text{for } \bx \in \Omega_2^{r_0} \\
   0, \quad &\text{for } \bx \in \Omega_3^{r_0}
\end{cases}.
$$
For $\bx \in \cup_{i=1}^3 \partial \Omega_i^{r_0}$, the classical representation formula (\cite[Thm.~3.1.6]{Sauter:2011}) yields
\begin{align}
\label{eq:Zrep}
  Z = \cS_{\partial \Omega_1^{r_0}} \gamma_N^1 Z_1 - \cS_{\partial \Omega_2^{r_0}} \gamma_N^2 Z_2- \mathcal{D} \mu,
\end{align}
and taking traces of $Z_i$ over $\partial \Omega_i^{r_0}$, for each $i\in\{1,2,3\}$, we obtain 
\begin{align*}
  \widetilde{\mu} = \gamma_D^1 \cS_{\partial \Omega_1^{r_0}} \gamma_N^1 Z_1- \gamma_D^1 \cS_{\partial \Omega_2^{r_0}} \gamma_N^2 Z_2- \gamma_D^1 \mathcal{D} \mu \quad\text{in }\partial\Omega_1^{r_0}, \\
    -\widetilde{\mu} = \gamma_D^2 \cS_{\partial \Omega_1^{r_0}} \gamma_N^1 Z_1- \gamma_D^2 \cS_{\partial \Omega_2^{r_0}} \gamma_N^2 Z_2- \gamma_D^2 \mathcal{D} \mu \quad\text{in }\partial\Omega_2^{r_0}, \\
    0 = \gamma_D^3 \cS_{\partial \Omega_1^{r_0}} \gamma_N^1 Z_1- \gamma_D^3 \cS_{\partial \Omega_2^{r_0}} \gamma_N^2 Z_2- \gamma_D^3 \mathcal{D} \mu \quad\text{in }\partial\Omega_3^{r_0},
\end{align*}
so that 
\begin{align}
  \langle \gamma_N V, \mu  \rangle_{\Gamma} &= \langle \gamma_N^1 V, \mu  \rangle_{\widehat\Gamma_1} -\langle \gamma_N^2 V, \mu  \rangle_{\widehat\Gamma_2}
  =\langle \gamma_N^1 V, \widetilde\mu  \rangle_{\partial\Omega^{r_0}_1} -\langle \gamma_N^2 V, \widetilde\mu  \rangle_{\partial\Omega^{r_0}_2}\nonumber\\
  &= \lim_{r \rightarrow \infty}\sum_{i=1}^3 \langle \gamma_N^i V, \gamma_D^i \cS_{\partial \Omega_1^{r_0}} \gamma_N^1 Z_1 \rangle_{\partial \Omega_i^r}- \langle \gamma_N^i V,\gamma_D ^i \cS_{\partial \Omega_2^{r_0}} \gamma_N^2 Z_2 \rangle_{\partial \Omega_i^r} + \langle\gamma_N^i V, \gamma_D^i
  \cD\mu\rangle_{\partial \Omega_i^r}.
  \label{eq:murep}
\end{align}
Since the Single-Layer potential is continuous across boundaries and $V$ is assumed to be smooth, we can conclude that the first two terms in \cref{eq:murep} are equal to zero, so that
\begin{equation}
\label{eq:DirJumpD}
 \langle \gamma_N V , \mu \rangle_{\Gamma} =\lim_{r \rightarrow \infty} \sum_{i=1}^3 \langle \gamma_N^i V, \gamma_D^i \cD\mu \rangle_{\partial \Omega_i^r} , 
\end{equation}
which, combined with \cref{eq:jumpsDL} and $U=-\cD\mu$, yields 
$$\lim_{r \rightarrow \infty} \sum_{i=1}^3 \langle \gamma_N^i U, \gamma_D^i V \rangle_{\partial \Omega_i^r}= 0,$$
which further implies that $[\gamma_N U ] = 0$.
Moreover, repeating the procedure in \cref{eq:murep} for the product $\langle \gamma_N U, \mu\rangle_{\Gamma}$, together with our result that the Neumann jump of $U$ is zero, yields
$$
\langle \mathcal{W} \mu, \mu \rangle_\Gamma = \langle \gamma_N U , \mu \rangle_{\Gamma} = \lim_{r\rightarrow \infty} \sum_{i=1}^3 \langle \gamma_N^i U, \gamma_D^i U\rangle_{\partial \Omega_i^r}. 
$$
Finally, by Green's formula and using that $U$ solves the Laplace equation on $\Omega_i^r$, we obtain 
\begin{align}
\label{eq:WsemDef}
  \lim_{r\rightarrow \infty} \sum_{i=1}^3 \langle \gamma_N^i U, \gamma_D^i U\rangle_{\partial \Omega_i^r} =
  \lim_{r\rightarrow \infty} \sum_{i=1}^3 \int\limits_{\Omega_i^r}\nabla U(\bx)\cdot\nabla U(\bx)\ \d\bx
  = \lim_{r\rightarrow \infty} \sum_{i=1}^3|U|^2_{H^1 (\Omega_i^r)} \geq 0. 
\end{align}
We proceed with showing that the hypersingular operator is injective.
Suppose that, for some given $\mu\in \widetilde{H}^{\half}(\Gamma)$, we have that $\mathcal{W} \mu = 0$.
Then, \cref{eq:WsemDef} allows us to deduce that $U$ is constant, and the decay properties of $U$ (stated in \Cref{lemma:DLdecay})
further allows us to deduce that $U=0$. Then, \cref{eq:Zrep} reduces to 
$$
Z = \cS_{\partial \Omega_1^{R_0}} \gamma_N^1 Z_1 - \cS_{\partial \Omega_2^{R_0}} \gamma_N^2 Z_2,
$$
which implies that $Z$ does not jump across $\Gamma$, which is only possible if $\mu = 0$. 
\end{proof}

Unfortunately, the previous result is not enough to show that $\mathcal{W} : \widetilde{H}^{\half}(\Gamma) \rightarrow (\widetilde{H}^{\half}(\Gamma))^\star$ is invertible. The missing ingredient to show the ellipticity of the hypersingular operator is a bound of the form
$$
\sum_{i=1}^n| U|_{H^1(\Omega_i^r)}^2 \gtrsim \Vert \mu \Vert^2_{\widetilde{H}^{\half}(\Gamma)}. 
$$
At first glance, it would seem that this could follow from \cref{eq:DirJumpD}, as we should have, for $V \in \C^{\infty}_0(\IR^2)$, that
\begin{equation}
\label{eq:failedbound}
    \langle \gamma_NV, \mu \rangle_\Gamma \lesssim \lim_{r \rightarrow \infty}\sum_{i=1}^n \Vert\gamma_N^i V \Vert_{H^{-\half}(\partial \Omega_i^r)} \Vert \gamma_D^i U \Vert_{H^{\half}(\partial \Omega_i^r)}.
\end{equation}
If we had that the space of Neumann traces of functions in $\C^\infty_0 (\IR^2)$ were dense in $\widetilde{H}^{\half}(\Gamma)^\star$,
we could take the supremum over all $V\in\C^{\infty}_0(\IR^2)$ with unit norm in $\widetilde{H}^{\half}(\Gamma)^\star$ to get, in the 
left hand side of \cref{eq:failedbound}, the norm of $\mu$ in $\widetilde{H}^\half(\Gamma)$. However, since this density does not hold, we are unable to connect \cref{eq:failedbound} with the $\widetilde{H}^\half(\Gamma)$ norm of $\mu$ so as to replicate the
arguments in, for example, \cite[Thm.~3.5.3]{Sauter:2011}.

\begin{remark}
Although the analysis has been presented for the triple junction geometry displayed in \Cref{fig:NeumanGeo},
the framework extends to general multi-arc configurations $\Gamma$ equipped with a normal field $\bm n$, given that the
following condition holds: there exists an index subset $\widetilde{I} \subset \{1,\hdots, n\}$ such that: 
\begin{enumerate}
  \item $\Gamma \subset \cup_{i \in \widetilde{I}} \partial \Omega_i$.
  \item If we denote the exterior normal vector of $\Omega_i$ as $\bm{n}_i$ for $i \in \widetilde{I}$, then either $$\bm{n}|_{\widehat{\Gamma}_i} = \bm{n}_i|_{\widehat{\Gamma}_i} \qquad\text{or}\qquad \bm{n}|_{\widehat{\Gamma}_i} = -\bm{n}_i|_{\widehat{\Gamma}_i},$$ for each $i \in \widetilde{I}$.
\end{enumerate}
\end{remark}

\section{Singular Behavior of Solutions to the Dirichlet Problem} \label{sec:SingDir}
The purpose of this section is to investigate the behavior of solutions to the
boundary integral equation associated with the Dirichlet problem introduced in 
\cref{eq:DirBIE}.  We will show that the classical analysis for singularities on polygonal boundaries \cite{costabel1985boundary, heuer1996hp} can be extended to an special geometry, the symmetric triple junction. For other geometries we have not been able to carry out the rigorous analysis and we will limit to present some hypothesis based on numerical experiments.

\subsection{Symmetric Triple Junction.}
\label{sec:striplejuct}
We consider a symmetric triple-junction geometry consisting of three straight line segments meeting at a common point (the junction center), with equal opening angles between consecutive segments, namely
$$\theta = \frac{2\pi}{3}.$$
As mentioned before, we will characterize the singular behavior of the solution to the boundary integral equation in \cref{eq:DirBIE} in a neighborhood of the junction point by extending the classical analysis developed for polygonal domains and open arcs (\emph{c.f.}~\cite{heuer1996hp,costabel1985boundary}).

Following \cite{heuer1996hp}, we embed the geometry in the complex plane $\IC$ and
idealize each segment as an infinite ray emanating from the origin. Without loss of
generality, we may choose on of the rays to coincide with the positive real
axis. Then, the following boundary integral equation
$$
\int\limits_\Gamma \frac{-1}{2 \pi } \log \Vert\bx - \by\Vert \lambda(\by)\ \d \by = f(\bx), 
$$
is essentially equivalent\footnote{Essentially equivalent in the sense that certain finite-dimensional correction operators are omitted. We refer to \cite{heuer1996hp} for further details.} to the system
\begin{align}
\label{eq:systIntEq}
\begin{pmatrix}
V_0 & V_\theta & V_{2 \theta} \\
V_\theta & V_0 & V_\theta \\
V_{2 \theta} & V_\theta & V_0
\end{pmatrix}
\begin{pmatrix}
  \lambda_0(x) \\ \lambda_1(x) \\
  \lambda_2(x ) 
\end{pmatrix}
= \begin{pmatrix}
f_0(x) \\
f_1(x ) \\
f_2(x )
\end{pmatrix} \quad x \geq 0,
\end{align}
where the integral operators are given by
$$V_\eta \psi = \frac{-1}{2\pi}\int\limits_{0}^\infty \log \left\lvert 1-\frac{y}{x}e^{i\eta}\right\rvert \psi(y)\ \d y, $$
for each $\eta\in\IR$, $\lambda_0(x) = \lambda(x), \lambda_1(x) = \lambda(xe^{i\theta}), \lambda_2(x) = \lambda(xe^{i 2 \theta})$, and the functions $f_0,f_1, f_2$ are defined analogously. Moreover, since $\theta = \frac{2\pi}{3}$ it holds that $V_{\theta} = V_{2\theta}$, which enables us to diagonalize the system in
\cref{eq:systIntEq} as
$$
\begin{pmatrix}
  -1 & -1 & \phantom{-}1 \\
  \phantom{-}1 & \phantom{-}0 & \phantom{-}1 \\
  \phantom{-}0 & \phantom{-}1 & \phantom{-}1 
\end{pmatrix}
\begin{pmatrix}
  V_0 - V_\theta & 0 & 0 \\
  0 & V_0 - V_\theta & 0 \\
  0 & 0 & V_0 + 2 V_\theta
\end{pmatrix}
\frac{1}{3}
\begin{pmatrix}
-1 & \phantom{-}2 &-1 \\
-1 & -1 & \phantom{-}2 \\
\phantom{-}1 & \phantom{-}1 & \phantom{-}1
\end{pmatrix}
\begin{pmatrix}
  \lambda_0(x) \\ \lambda_1(x) \\
  \lambda_2(x ) 
\end{pmatrix}
 = 
 \begin{pmatrix}
f_0(x) \\
f_1(x ) \\
f_2(x )
\end{pmatrix},
$$
or, equivalently, 
$$
\begin{pmatrix}
  V_0 - V_\theta & 0 & 0 \\
  0 & V_0 - V_\theta & 0 \\
  0 & 0 & V_0 + 2 V_\theta
\end{pmatrix}
\begin{pmatrix}
  \sigma_0(x) \\
  \sigma_1(x) \\
  \sigma_2(x)
\end{pmatrix}
= \begin{pmatrix}
  g_0(x)\\
  g_1(x) \\
  g_2(x)
\end{pmatrix},
$$
where 
$\sigma_0 = \frac{1}{3}(2\lambda_1 -\lambda_0-\lambda_2)$, $\sigma_1 = \frac{1}{3}(2 \lambda_2-\lambda_0-\lambda_1)$, and $\sigma_2 = \frac{1}{3}(\lambda_0+\lambda_1+\lambda_2)$ and the functions $g_0$, $g_1$ and $g_2$ are analogously defined from $f_0$, $f_1$ and $f_2$. Moreover, since the operators in the previous system are the same as in \cite{heuer1996hp}, it follows from the cited work that by taking the Mellin transform we get that
$$
{\scriptstyle 
\begin{pmatrix}
  \frac{\cosh(\pi s)}{s\sinh(\pi \lambda)} - \frac{\cosh((\pi-\theta) s)}{s\sinh(\pi s)} & 0 & 0 \\
  0 &   \frac{\cosh(\pi s)}{s\sinh(\pi s)} - \frac{\cosh((\pi-\theta)s)}{s \sinh(\pi s)} & 0 \\
  0 & 0 &   \frac{\cosh(\pi s)}{s\sinh(\pi s)} +2 \frac{\cosh((\pi-\theta)s)}{s\sinh(\pi s)}
\end{pmatrix}
\begin{pmatrix}
  \widehat{\sigma}_0(s-i) \\
  \widehat{\sigma}_1(s-i) \\
  \widehat{\sigma}_2(s-i)
\end{pmatrix}
= \begin{pmatrix}
  \widehat{g}_0(s)\\
  \widehat{g}_1(s) \\
  \widehat{g}_2(s)
\end{pmatrix},}
$$
where $\widehat{g}_j$ and $\widehat{\sigma}_j$ denote the respective Mellin transforms of $g_j$ and $\sigma_j$ for $j\in\{0,1,2\}$ and  $s \in \IC$ with $\Im({s})\in(0,1)$. The orders of the singularities of $\sigma_0$ and $\sigma_1$
are then determined through the poles of
$$
  \frac{(s+i)\sinh(\pi (s+i))}{\cosh(\pi(s+i))-\cosh((\pi-\theta)(s+i))} \widehat{g}_j(s+i),
$$
for $j\in\{0,1\}$, while the order of the singularities of $\sigma_2$ are determined by the poles of
$$
  \frac{(s+i)\sinh(\pi (s+i))}{\cosh(\pi(s+i))+2\cosh((\pi-\theta)(s+i))} \widehat{g}_2(s+i).
$$
The detailed computation of these poles is provided in Appendix \ref{app:compPoles}, from where we derive that
$$
\sigma_j(x) = \sum_{k=1}^{n} (\sigma_j^1)_k x^{3k-\frac{5}{2}}+(\sigma_j^2)_k x^{3k-1}+(\sigma_j^3)_k x^{3k-1} \log|x|+\sigma_j^0(x),
$$
for $j\in\{0,1\}$ and
$$
\sigma_2(x) = \sum_{k=0}^n (\sigma_2^1)_k x^{\half+3k}+
(\sigma_2^2)_k x^{3k}+(\sigma_2^3)_k x^{3k+1}+ \sigma_2^0(x), 
$$
where $\{(\sigma_j^n)_k\}_{k \in \IN_0}$ are the coefficients of each expansion and $\sigma_j^0(x)$ are functions that could only have weaker singularities, as detailed in \cite{heuer1996hp}.

These representations allow us to recover the asymptotic behavior of $\lambda_0$, $\lambda_1$, and $\lambda_2$. In particular, one finds that the leading singular term behaves as $x^{\half}$ near the triple junction, which coincides with the corresponding singularity for two segments meeting at an angle of $\frac{2\pi}{3}$.

\subsection{Galerkin Discretization.}
\label{sec:Discretrization}
The previous strategy relies heavily on the diagonalization of the system in \cref{eq:systIntEq} and is, therefore, not easily extended to more general arrangements. A study of the singular behavior of solutions to the Dirichlet problem is beyond the scope of the present work. Therefore, we investigate more general geometries through numerical simulations only. Consequently, we continue by constructing discrete approximations of the solution to \cref{eq:Veq} which,
according to \Cref{thrm:uniquesol}, require finite dimensional approximations of $\widetilde{H}_*^{-\half}(\Gamma)$. However, the problem may be reformulated in $\widetilde{H}^{-\half}(\Gamma)$ by means of the following mixed formulation: seek $\lambda \in \widetilde{H}^{-\half}(\Gamma)$ and $\alpha \in \IR$ such that
\begin{align}
\label{eq:modsystem}
\begin{aligned}
  \langle  \varphi, \mathcal{V} \lambda \rangle_\Gamma + \alpha \langle  1, \varphi\rangle_\Gamma &= \langle f, \varphi \rangle_{\Gamma}, \quad \forall \varphi \in \widetilde{H}^{-\half}(\Gamma), \\
  \langle 1,\lambda  \rangle_\Gamma &= 0.
  \end{aligned}
\end{align}
We refer to \cite[Sec.~3.5]{STEI07} regarding the equivalence of both formulations. 
Moreover, we will assume that the arcs which form part of $\Gamma$, namely $\Gamma_1, \ldots, \Gamma_n$, are parametrized by analytic mappings
\[
\br_1, \ldots, \br_n : [-1,1] \rightarrow \IR^2,
\]
each of which has a nowhere-vanishing tangent vector. 
We also assume that, for every $i \in\{1, \ldots, n\}$, the endpoint $\br_i(-1)$ is the only point of the arc $\Gamma_i$ that belongs to $\partial \Gamma$. With these assumptions at hand, we continue by introducing two families of conforming discrete spaces, which will be used to approximate the solution of the integral equation associated with the Dirichlet problem through the standard Galerkin method.


\subsubsection{Spectral Discretization}
\label{sec:SpectralImplementation}
For $N \in \IN$, we denote by $P_N$ the space of polynomials of degree at most $N$ on $[-1,1]$. Then, for any $\alpha, \beta \in [-\half,1)$ we introduce the following weighted polynomial space:
\begin{align*}
    P_N^{\alpha, \beta} := \{(1-t)^\alpha (1+t)^\beta p_N(t)\ :\ p_N\in P_N\},
\end{align*}
as well as the parametrized spaces
\begin{align*}
P^{\alpha,\beta}_{N,i} =\left\lbrace p_N^{\alpha,\beta} \circ \br_i^{-1} : \ p_N^{\alpha,\beta} \in P_N^{\alpha, \beta}\right\rbrace,
\end{align*}
for each $i\in\{1,\hdots,n\}$.
Finally, for $\bm{\alpha}, \bm{\beta} \in [-\half,1)^n$ we define the discrete space 
$$
\IP_N^{\bm{\alpha }, \bm{\beta}} := \left \lbrace p : \Gamma \rightarrow \IC \ : 
p|_{\Gamma_i} \in P^{\alpha_i,\beta_i}_{N,i}\ \forall i \in\{1,\hdots, n\}
\right\rbrace.
$$
The presence of the weight $(1-t)^\alpha (1+t)^\beta$ will allow us to capture potential singularities that may appear at the endpoints of each arc $\Gamma_i$. Based on previous results for open arcs (\emph{c.f.}~\cite{JHP20}), we expect that an appropriate selection of $\bm{\alpha},\bm{\beta}$ leads to improved convergence rates. However, there are some additional complications compared to the open arc case that will be discussed at the end of this section. 

For the numerical implementation of the spectral discretization spaces introduced above, we will extend the methods presented in \cite{JHP20}. In this previous work we considered integrals of the forms
\begin{align}
\label{eq:intscanonical}
\widehat{I}^1_{l,m}(i,j) &= 
\int\limits_{-1}^1 \int\limits_{-1}^1 G(\br_{i}(t)-\br_j(s)) \frac{T_m(s)}{w(s)}\frac{T_l(t)}{w(t)}\ \d s\ \d t, \\
\widehat{I}^2_{l}(i) &= 
\int\limits_{-1}^1 f(\br_i(t))\frac{T_l(t)}{w(t)}\ \d t,
\end{align}
where $\{T_k\}_{k=1}^N$ denotes the first kind Chebishev polynomials of first kind, $l,m\in\{0,\hdots, N\}$, $w(t) = \sqrt{1-t^2}$ and $\br_i$ and $\br_j$ were assumed parametrization of two disjoint arcs (if $j \neq i$ ).
Then these two type of  integrals were approximated using an FFT based algorithm, which leads to an algorithm that converges exponentially in the number of function evaluations. 

 Moreover, we showed based on \cite[Chap.~6,7]{trefethen2013approximation} that $\widehat{I}^2_l(i)$ decays exponentially with respect to $l$. Which we later use to establish the exponential convergence of the spectral method in \cite{JHP20}. 

For the multi-arc setting, we choose the Chebishev polynomials as a basis of $P_N$, thus the discretization of the corresponding linear and bilinear forms in \cref{eq:modsystem} leads to integrals of the form
\begin{align*}
{I}^1_{l,m}(i,j) &= 
\int\limits_{-1}^1 \int\limits_{-1}^1 G_{i,j,\bm{\alpha},\bm{\beta}}(t,s) \frac{T_m(s)}{w(s)}\frac{T_l(t)}{w(t)}\ \d s\ \d t, \\
{I}^2_{l}(i) &= 
\int\limits_{-1}^1 f_{i,\bm{\alpha},\bm{\beta}}(t)\frac{T_l(t)}{w(t)}\ \d t,
\end{align*}
for all $l,m\in\{0,\hdots,N\}$, where
\begin{align*}
G_{i,j,\bm{\alpha},\bm{\beta}}(t,s) &\coloneq G(\br_i(t)-\br_i(s)) p_{\alpha_i,\beta_i}(t)p_{\alpha_j,\beta_j}(s),\\
f_{i,\bm{\alpha},\bm{\beta}}(t) &\coloneq f(\br_i(t)) p_{\alpha_i,\beta_i}(t),
\end{align*}
with $p_{\alpha,\beta}(t) = (1-t)^{\alpha+\half}(1+t)^{\beta+\half}$, and $i,j \in \{1,\hdots,n\}$.

The approximation of these type of integrals is obtained exactly as in \cite{JHP20}, but with $f_{i,\bm{\alpha},\bm{\beta}}$ and $G_{i,j,\bm{\alpha},\bm{\beta}}$ taking the roles of $f$ and $G$ in \eqref{eq:intscanonical}. 

The change on the functions $f, G$ implies that in contrast to the case in \cite{JHP20},
 if $\alpha_i,\beta_i \neq \pm \frac{1}{2}$, the functions $f_{i,\bm{\alpha},\bm{\beta}}$ and $G_{i,j,\bm{\alpha},\bm{\beta}}$, are not smooth. This has two implications:

\begin{enumerate}
    \item The overall approximation process do not converge exponentially in the number of function evaluations. 

    This could in theory be remedied by replacing the FFT approximation method by the Jacobi quadrature rules \cite{HaleTownsend2013}.
    Furthermore, if the the Jacobi polynomials are selected as basis of $P_N$ instead of the Chebishev polynomials we would obtain the Galerkin equivalent of the method in \cite{HOUGH1985359}. However for the propose of this work this do not alter the result in any relevant manner. 

    \item 
    The coefficients $I_l^2$ do not longer decay exponentially with $l$.

    This is in fact a fundamental property of the coefficients and not of the approximation method. Moreover if we follow our previous analysis in \cite{JHP20} and \cite{Tao22}, the exponential decay of these coefficients was the essential argument to stablish the corresponding convergence of the spectral method. 
    
    As consequence we not longer expect the spectral method to converge exponentially, but instead to have a algebraic rate which is determined by the decay of $I_l^2$.     
\end{enumerate}

Finally, if we notice that since now the arcs parametrized by $\br_i,\br_j$ are no longer disjoint, additional logarithmic singularities could occur in the computation of $I^1_{l,m}$. However, they are balanced by the behavior of $p_{\alpha,\beta}(t)$ near the end points. 

\subsubsection{Low-Order Discretization}
\label{sec:loworderDisc}

Given $N \in \IN$ and $a, b \geq 1$, we consider two graded meshes of the interval $[0,1]$
with nodes $\{\hat{x}^1_k\}_{k=0}^N$ and $\{\hat{x}^2_k\}_{k=0}^N$ given as
$$
\hat{x}^1_k:=\left(\frac{k}{N}\right)^a\quad\text{and}\quad \hat{x}^2_k:=\left(\frac{k}{N}\right)^b, \quad \forall k \in\{0,\hdots, N\}. 
$$
As seen from the previous definitions, both meshes concentrate nodes near $0$. We then map these points to the intervals $[0,1]$ and $[-1,0]$, respectively, in such a way that they are refined towards $1$ (with grading factor $a$) and $-1$ (with grading factor $b$). The resulting mesh of $[-1,1]$ is then denoted as $\widehat{\mathcal{T}}^{a,b}_N$. We then define the corresponding meshes on the arcs $\Gamma_1, \hdots, \Gamma_n $ as 
\begin{align*}
\mathcal{T}^{a,b}_{N,i} := \{\tau\ :\ \tau = \br_i(\hat\tau),\ \forall\hat\tau\in\widehat{\mathcal{T}}^{a,b}_N\}
\end{align*}
for each $i\in\{1,\hdots,n\}$, and denote by $\cS^{a,b}_{N,i}$ the space spanned by piecewise constant functions on the elements of $\mathcal{T}^{a,b}_{N,i}$.
Finally, for $\bm{a}, \bm{b} \in [0,1]^n$ we define the discrete space 
$$
\mathbb{S}^{\bm{a},\bm{b}}_N := \left\lbrace s:\Gamma \rightarrow \IR \ : s|_{\Gamma_i} \in \cS^{a_i, b_i}_{N,i}, \ i = 1,\hdots,n\right\rbrace.
$$
The factors $\bm{a}$ and $ \bm{b}$ play analogous roles to those of the parameters $\bm{\alpha}$ and $\bm{\beta}$ in the spectral discretization, in the sense that they allow us to properly capture the singular behavior of the solution near the endpoints and junctions (we refer to \cite[Chapter~7.1]{gwinner2018advanced} for further details). In particular, singularities of the form $t^r$ with $0<r<1$, and $t$ in a neighborhood of $0$, are well captured by using a mesh whose refinement parameter near $0$ is larger than $\frac{3}{2(r+1)}$.
An appropriate choice of the parameters $\bm{a}$ and $\bm{b}$ should allow us to recover the optimal convergence rate of $N^{-\threehalf}$ for the error measured in the norm of $\widetilde{H}^{-\half}(\Gamma)$.

\subsection{Numerical Experiments}
\label{sec:NumRes}

In this section, we analyze the numerical solution of the Dirichlet problem for several geometries and consider Galerkin discretizations of \cref{eq:modsystem} using both the spectral and low-order schemes. In each case, we measure the error in the $\widetilde{H}^{-\half}(\Gamma)$ norm with respect to an \emph{overkill} (highly resolved) solution. 
For the right-hand side $f$, we \ra{choose} the function
$$
f(\bx) = \gamma_D(2x_1 + x_2),
$$
and as test geometries, we primarily focus on triple junctions as is \Cref{fig:triplejunction}. The explicit parametrizations of the arcs, in this case, are
\begin{equation}
    \label{eq:param}
    \begin{aligned}
  \br_1(t) &= \left( 0, \frac{t-1}{2} \right) \\
  \br_2(t) &= \frac{1-t}{\sqrt{2}}\left(\cos \theta_1, \sin \theta_1  
  \right)\\
   \br_3(t) &= \frac{1-t}{\sqrt{2}}\left(\cos \theta_2, \sin \theta_2  
  \right),
\end{aligned}
\end{equation}
and we refer to \Cref{fig:triplejunction} for a clear visualization of the geometry and the meaning of the parameters $\theta_1$ and $\theta_2\in (0,\frac{3}{2}\pi)$.

Additionally, in \Cref{sec:dt}, we also consider a more general (\emph{Double Triple Junction}) geometry, visible in \Cref{fig:dtGeo}. 

For the implantation of the low-order scheme, we rely on the Gypsilab library \cite{alouges2018fem}, while the implementation of the spectral discretization was carried out in Matlab.

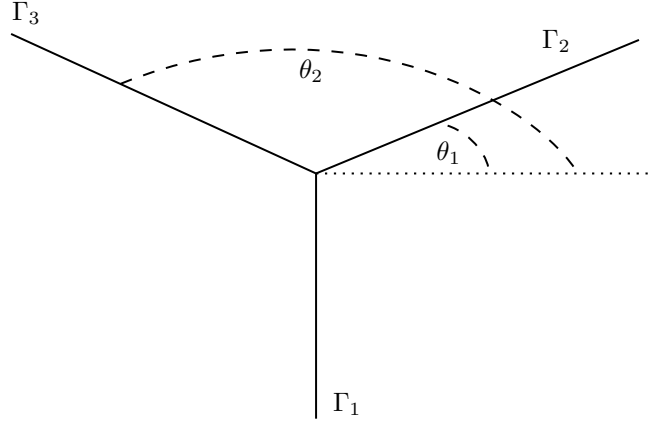
\begin{figure}
  \centering

\tikzset{every picture/.style={line width=0.75pt}} 

\begin{tikzpicture}[x=0.75pt,y=0.75pt,yscale=-1,xscale=1]

\draw  (317,126) -- (317,249) ;
\draw  (317,126) -- (479,59) ;
\draw  (164,56) -- (317,126) ;
\draw [dash pattern={on 0.84pt off 2.51pt}] (317,126) -- (490,126) ;
\draw [dash pattern={on 4.5pt off 4.5pt}] (383,102) .. controls (395.5,106) and (404.5,119) .. (403.5,126) ;
\draw [dash pattern={on 4.5pt off 4.5pt}] (219,81) .. controls (333,35) and (429,92) .. (449,127) ;

\draw (376,108.4) node [anchor=north west][inner sep=0.75pt]  {$\theta _{1}$};
\draw (307,67.4) node [anchor=north west][inner sep=0.75pt]  {$\theta _{2}$};
\draw (324,234.4) node [anchor=north west][inner sep=0.75pt]  {$\Gamma _{1}$};
\draw (429,52.4) node [anchor=north west][inner sep=0.75pt]  {$\Gamma _{2}$};
\draw (163,38.4) node [anchor=north west][inner sep=0.75pt]  {$\Gamma _{3}$};

\end{tikzpicture}






  \caption{Triple junction geometry.}
  \label{fig:triplejunction}
\end{figure}

\subsubsection{Symmetric Triple Junction}

We consider the triple junction geometry described through the parametrization in \cref{eq:param}, and begin by choosing the parameter values as $\theta_1 = \frac{\pi}{6}$ and $\theta_2 = \frac{5\pi}{6}$, leading to the symmetric triple junction as described in \Cref{sec:striplejuct}.
As previously discussed, the solution exhibits a singularity of order $\half$ near the triple point, while near the endpoints of the geometry it exhibits the same singular behavior as for an open arc, namely of order $-\half$.

This observation implies that, for the spectral discretization, we should select $\bm{\alpha } = \big( \alpha,\alpha, \alpha \big)$, with $\alpha = \half$ and $\bm{\beta} = \big( \frac{-1}{2},\frac{-1}{2},\frac{-1}{2}\big)$. For the low-order discretization, we choose $\bm{a} = (a,a,a)^\top,$ with $a >1$, and $\bm{b} = (b,b,b)^\top$ with $b > 3$. 

In \Cref{fig:symetricerror}, we report the errors of the spectral method for different values of $\alpha$. As predicted, the best convergence rates are obtained for $\alpha = \half$, and also $\alpha = -\half$. The latter is consequence of the fact that if the solution behaves as $t^{\half}\psi(t)$, with $\psi$ a smooth function, then it can equivalently be written as $t^{-\half}\widehat{\psi}(t)$, where $\widehat{\psi}$ is also smooth. 

\begin{figure}
  \centering
    \includegraphics[width=0.5\linewidth]{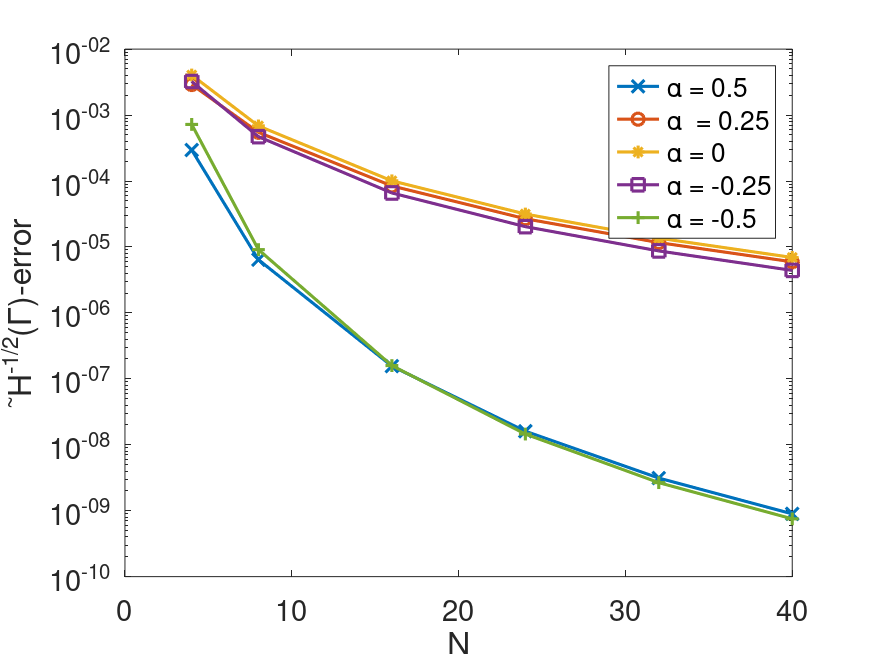}
  \caption{Convergence of the spectral method for the symmetric triple junction. Error measured in the $\widetilde{H}^{-\half}(\Gamma)$-norm, with respect to an overkill solution obtained using $N= 80$. }
  \label{fig:symetricerror}
\end{figure}

On \Cref{tab:p0T}, we present the results for the convergence of the the low-order discretization. As predicted, the convergence rate is close to $-\threehalf$ when $a >1$ and $b >3$.

\begin{table}[ht]
\begin{tabular}{ll|ll|ll}
\multicolumn{2}{l|}{$a= 1.1, b= 1.1$}                 & \multicolumn{2}{l|}{$a = 1.1, b= 2.1$}                & \multicolumn{2}{l}{$a = 1.1, b= 3.1$}                  \\ \hline
\begin{tabular}[c]{@{}l@{}}Points\\ per arc\end{tabular} & Error  & \begin{tabular}[c]{@{}l@{}}Points\\ per arc\end{tabular} & Error  & \begin{tabular}[c]{@{}l@{}}Points\\ per arc\end{tabular} & Error  \\
8                            & 0.9796 & 8                            & 0.591097 & 8                            & 0.446410 \\
16                            & 0.6075 & 16                            & 0.245440 & 16                            & 0.140612 \\
24                            & 0.4834 & 24                            & 0.164593 & 24                            & 0.080432 \\
32                            & 0.3830 & 32                            & 0.115673 & 32                            & 0.055012 \\
40                            & 0.3502 & 40                            & 0.092370 & 40                            & 0.036689 \\
48                            & 0.3282 & 48                            & 0.077846 & 48                            & 0.031226 \\
56                            & 0.2852 & 56                            & 0.066889 & 56                            & 0.024451 \\ \hline
Convergence rate                     & -0.6226 & Convergence rate                     & -1.1150 & Convergence rate                     & -1.4789 
\end{tabular}
\caption{Low-order discretization of the symmetric triple junction, errors measured in the $\widetilde{H}^{-\half}(\Gamma)-$norm with respect to a solution obtained using 256 points per arc.}
\label{tab:p0T}
\end{table}

\subsubsection{Unsymmetrical Triple Junction} 

Let us again consider the parametrizations given in \cref{eq:param}. We analyze two different configurations:
\begin{enumerate}
  \item $\theta_1= 0, \theta_2 = \pi$.
  \item $\theta_1 =(1-0.42)\pi-\frac{\pi}{2}, \theta_2 = (1+0.42)\pi-\frac{\pi}{2}$.  
\end{enumerate}
The first case corresponds to a $T$-shaped configuration. According to the classical analysis of singularities for polygons \cite[Chap.~7]{gwinner2018advanced}, we should expect singular behaviors of the forms $t^0, t^1$ in the central point, since the angles between the segments are either $\pi$ or $\frac{\pi}{2}$. This implies that the optimal convergence rate of $N^{-\frac{3}{2}}$ for the low-order discretization should be achieved by selecting (using the same notation as in the symmetric case) $a > 1$ and $b > 3$. The corresponding results are presented in \Cref{tab:errorUnCase1}.

Another interesting property that we can observe is the behavior $t^1$ (near the central point) is only present in $\Gamma_1$, while in $\Gamma_2$, and $\Gamma_3$ the solution behaves as $t^0$. We corroborate this observation by examining the behavior of the numerical solution as a function of the parameter $t$ in the parametrization in \Cref{fig:sollowordercase1}. 

\begin{table}[ht]
\centering
\begin{tabular}{ll}
\multicolumn{2}{l}{$a = 1.1, b= 3.1$}                  \\ \hline
\begin{tabular}[c]{@{}l@{}}Points\\ per arc\end{tabular} & Error  \\
8                            & 0.439238 \\
16                            & 0.137821 \\
24                            & 0.079154 \\
32                            & 0.055012 \\
40                            & 0.036896 \\
48                            & 0.032113 \\
56                            & 0.025528 \\ \hline
Convergence rate                     & -1.4622 
\end{tabular}
\caption{Low-order discretization of the triple junction with angles $\theta_1 = 0, \theta_2 =\pi$, errors measured in the $\widetilde{H}^{-\half}(\Gamma)$ norm with respect to a solution obtained using 256 points per arc.}
\label{tab:errorUnCase1}
\end{table}

\begin{figure}
  \centering
  \includegraphics[width=0.5\linewidth]{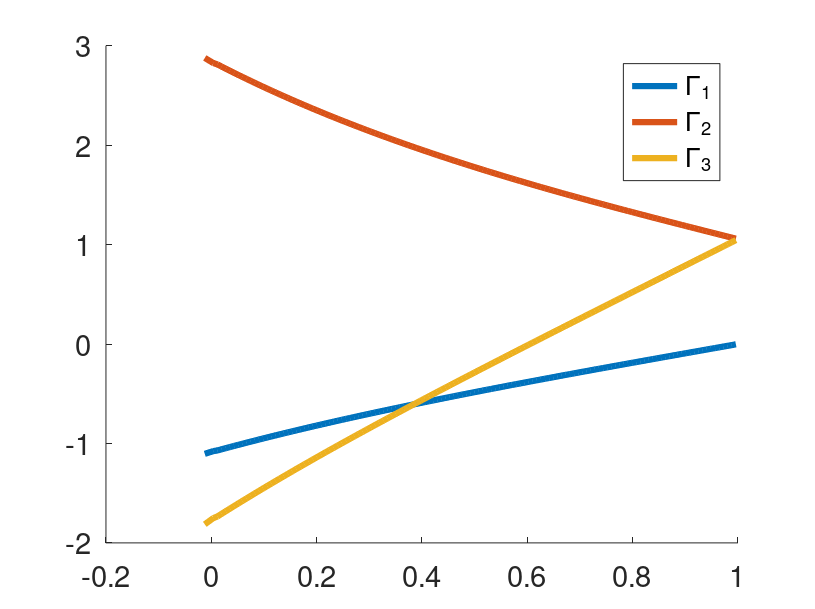}
  \caption{Plot of the solution for $\theta_1 = 0,\theta_2 = \pi $, for $t>0$}
  \label{fig:sollowordercase1}
\end{figure}

We can further corroborate the predicted behavior using the spectral discretization. For this configuration, the best convergence rate is expected when selecting $\alpha = 0$, which is confirmed by the results presented in \Cref{fig:spectralunsymcase1}. We also observe that the errors saturate more rapidly than in the symmetric case. We attribute this behavior to the fact that when $\alpha_i \neq \pm \half$, the integrand involves functions with unbounded derivatives, as discussed in \Cref{sec:SpectralImplementation}, which significantly amplifies quadrature errors.

\begin{figure}
  \centering
  \includegraphics[width=0.5\linewidth]{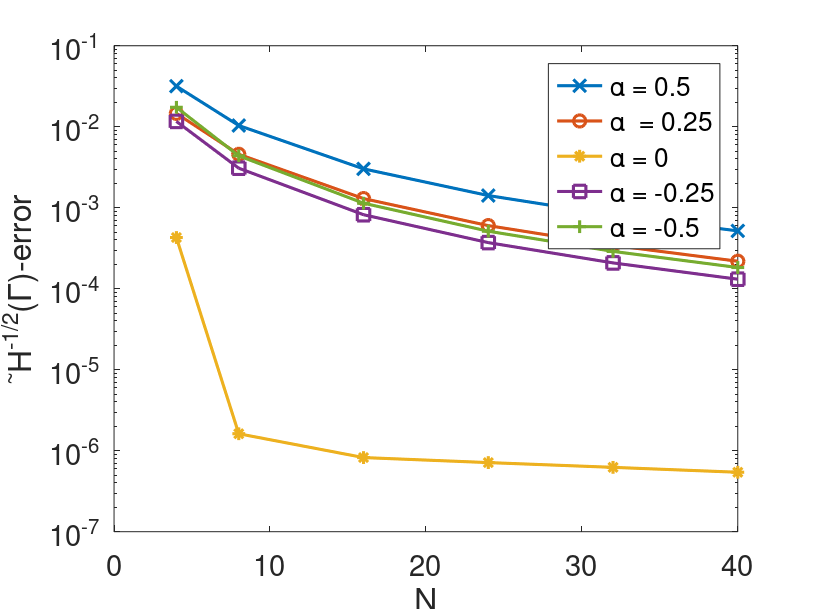}
  \caption{Convergence of the spectral method for the unsymmetrical case, with $\theta_1= 0, \theta_2= \pi$. Error measured in the $\widetilde{H}^{-\half}(\Gamma)$-norm, with respect to an overkill solution obtained using $N= 80$}
  \label{fig:spectralunsymcase1}
\end{figure}

For the second configuration, the predicted singular behavior near the central point, based on the theory of singularities for polygonal domains, should be of the form $t^{\frac{21}{29}}$, or $t^{\frac{4}{21}}$. Based on the results for the previous case, we hypothesize that the singularity of order $\frac{21}{29}$ occurs on $\Gamma_1$, while on $\Gamma_2,\Gamma_3$ the singularity is of the order $\frac{4}{21}$. Consequently, the spectral method is expected to achieve its best convergence rate by selecting $\alpha = (\frac{21}{29},\frac{4}{21},\frac{4}{21})$, and the corresponding results are presented in \Cref{fig:spectralunsymcase2}. For this configuration, we do not present results using the low-order scheme, as does not provide sufficient resolution to yield meaningful information.

\begin{figure}
  \centering
  \includegraphics[width=0.5\linewidth]{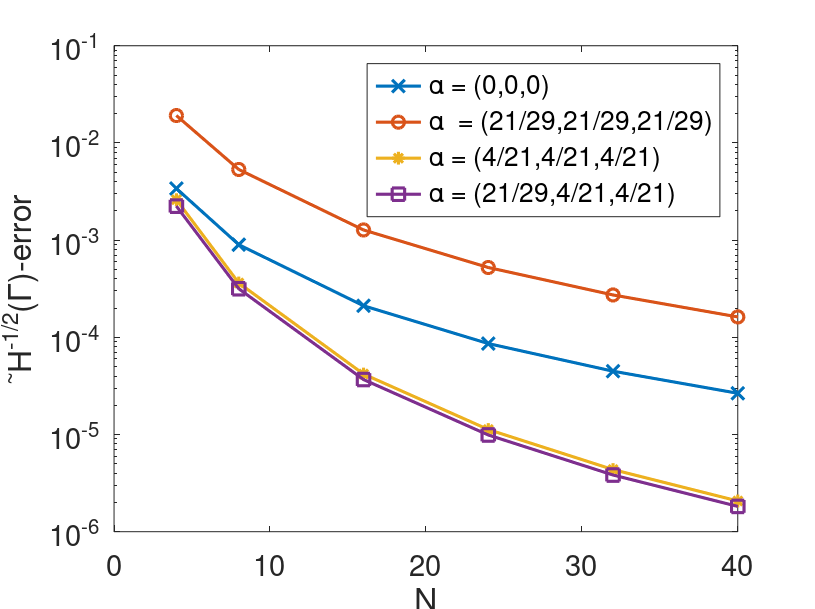}
  \caption{Convergence of the spectral method for the unsymmetrical case, with $\theta_1= (1-0.42)\pi-\frac{\pi}{2}, \theta_2= (1+0.42)\pi-\frac{\pi}{2}$.  Error measured in the $\widetilde{H}^{-\half}(\Gamma)$-norm, with respect to an overkill solution obtained using $N= 80$.}
  \label{fig:spectralunsymcase2}
\end{figure}

We observe that, in this last case, the difference between choosing $\alpha = (\frac{4}{21},\frac{4}{21},\frac{4}{21})$ and $\alpha = (\frac{21}{29},\frac{4}{21},\frac{4}{21})$ is relatively small. This can be explained by the following computation 
$$t^{\frac{21}{29}} = t^{\frac{4}{21}}t^{\frac{441}{116}},$$
and since $\frac{441}{116} \approx 3.8 \approx 4$ the singularities of the form $t^{\frac{21}{29}}$ can be accurately approximated by polynomials multiplied by a factor of $t^{\frac{4}{21}}$. 

Based on these results, we conjecture that the optimal choice of parameters for the spectral method is governed by the smallest (in magnitude) singular exponent. A more concrete conjecture is elaborated in \Cref{sec:conclusions}.

\subsubsection{Double Triple Junction}
\label{sec:dt}
As a final example, we consider a double triple-junction geometry, illustrated in \Cref{fig:dtGeo}. The exact parametrizations are given by
\begin{figure}
  \centering  \includegraphics[width=0.5\linewidth]{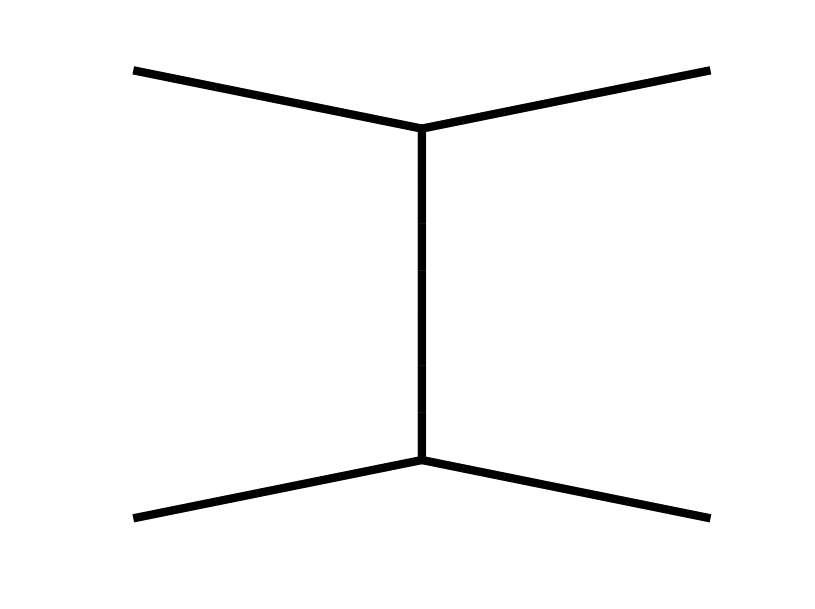}
    \caption{Double triple junction geometry.}
    \label{fig:dtGeo}
\end{figure}
\begin{align*}
  \br_1(t) &= (0,t), \\
  \br_2(t) &= \frac{1-t}{\sqrt{2}}(\cos \theta_1 , \sin \theta_1+1), \\
    \br_3(t) &= \frac{1-t}{\sqrt{2}}(\cos \theta_2 , \sin \theta_2+1),\\
     \br_4(t) &= \frac{1-t}{\sqrt{2}}(\cos \theta_3 , \sin \theta_3-1), \\
    \br_5(t) &= \frac{1-t}{\sqrt{2}}(\cos \theta_4 , \sin \theta_4-1).  
\end{align*}
We fix the angles as $\theta_1 = (1-0.42)\pi-\frac{\pi}{2}, \theta_2 = (1+0.42)\pi-\frac{\pi}{2}, \theta_3 = -\theta_1, \theta_4 = - \theta_2$. 
In this experiment, we compare the convergence of the spectral method for two different choices of parameters. First, we consider a non-singular selection
$$\bm{\alpha} = \bm{0}, \bm{\beta} = (0,-\half,-\half,-\half,-\half).$$
Second, we use the presumably optimal choice
$$\bm{\alpha} = (\frac{21}{29},\frac{4}{21},\frac{4}{21},\frac{4}{21},\frac{4}{21}), \quad \bm{\beta} =(\frac{21}{29},-\half,-\half,-\half,-\half).$$ The corresponding convergence results are presented in \Cref{fig:dtConvg}. 

\begin{figure}
  \centering  \includegraphics[width=0.5\linewidth]{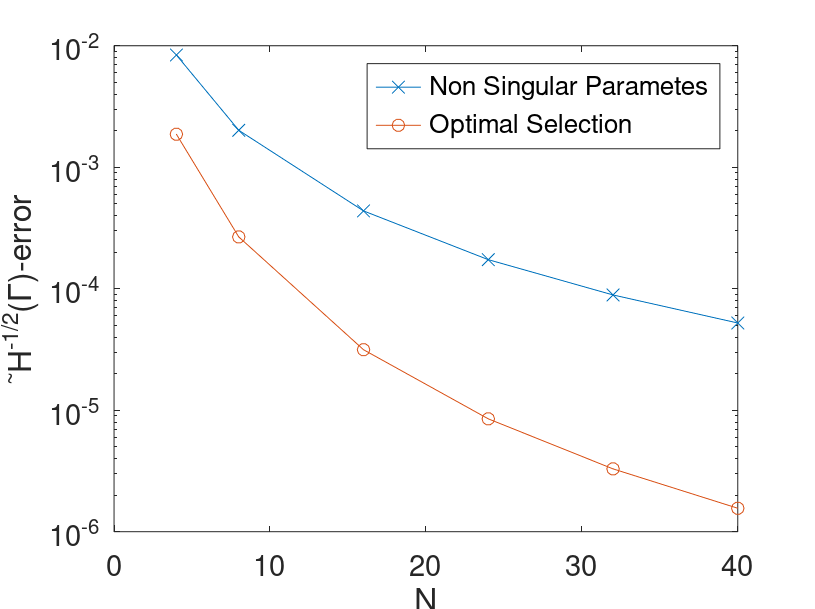}
    \caption{Convergence of the spectral method for the double triple junction. }
    \label{fig:dtConvg}
\end{figure}


\section{Numerical Experiments for the Neumann Problem}
\label{sec:NeumannNumeric}

We now consider the numerical discretization of the integral equation
$$
\mathcal{W} \mu = f. 
$$
Its variational formulation is 
$$
\langle -\gamma_N \mathcal{D} \mu, v \rangle_\Gamma = \langle f, v \rangle_\Gamma, \quad v \in \widetilde{H}^{\half}(\Gamma). 
$$
We can further expand the left-hand side using the corresponding definitions to obtain
\begin{equation}
\begin{split}
    \langle -\gamma_N \mathcal{D} \mu, v \rangle_\Gamma =
\langle -\gamma_N^1 \mathcal{D}_{\jp{\widehat{\Gamma}_1}} \mu, v\rangle_{\jp{\widehat{\Gamma}_1}} - \langle -\gamma_N^2 \mathcal{D}_{\jp{\widehat{\Gamma}_1}} \mu, v\rangle_{\jp{\widehat{\Gamma}_2}}-\\\langle -\gamma_N^1 \mathcal{D}_{\jp{\widehat{\Gamma}_2}} \mu, v\rangle_{\jp{\widehat{\Gamma}_1}}+\langle -\gamma_N^2 \mathcal{D}_{\jp{\widehat{\Gamma}_2}} \mu, v\rangle_{\jp{\widehat{\Gamma}_2}}. 
\end{split}
\end{equation}
All the terms on the right-hand side of the latter equation are hypersingular operators on open arcs and can be regularized using the standard integration by parts formula \cite[Thm.~6.15]{STEI07} \footnote{The regularization can be used on open arcs if we assume that the test and trial functions are 0 at the end points of the corresponding open arcs.}.

For the discretization space, we restrict ourselves to a low-order scheme. Under the notation of \Cref{sec:loworderDisc}, we denote by $\mathcal{P}_{N,i}^{a,b}$ the space of piecewise linear and globally continuous functions on $\mathcal{T}_{N,i}^{a,b}$. For $\bm{a},\bm{b} \in [0,1]^n$ we define the discretization space: 
$$
\mathcal{P}_{N}^{\bm{a},\bm{b}} \coloneqq \left\lbrace
p \in \C^0(\Gamma) : p|_{\Gamma_i} \in \mathcal{P}_{N,i}^{a,b}, \ p(\partial \Gamma) = 0
\right\rbrace. 
$$
The implementation is again carried out using the Gypsilab library \cite{alouges2018fem}. We consider exclusively the symmetric triple junction of \Cref{sec:striplejuct}. 

For the first numerical experiment, we consider the right-hand side $$f(x_1,x_2) = \gamma_N( \frac{|x_1|x_1^3}{4}).$$ 
The results are provided in \Cref{tab:p11}. We observe that convergence rates greater than $2$ in the $L^2$-norm are achieved when refinement is applied toward the endpoints of the multi-arc. Furthermore, no special refinement is required at the branch point.

\begin{table}[ht]
\begin{tabular}{ll|ll|ll}
\multicolumn{2}{l|}{$a = 1.1, b= 1.1$}                & \multicolumn{2}{l|}{$a = 1.1, b= 3.1$}                & \multicolumn{2}{l}{$a = 3.1, b= 3.1$}                \\ \hline
\begin{tabular}[c]{@{}l@{}}Points\\ per arc\end{tabular} & L2- Error & \begin{tabular}[c]{@{}l@{}}Points\\ per arc\end{tabular} & L2- Error & \begin{tabular}[c]{@{}l@{}}Points\\ per arc\end{tabular} & L2- Error \\
8                            & 0.071800 & 8                            & 0.017621 & 8                            & 0.019603 \\
16                            & 0.027030 & 16                            & 0.002965 & 16                            & 0.003151 \\
32                            & 0.011293 & 32                            & 0.000613 & 32                            & 0.000643 \\
64                            & 0.004727 & 64                            & 0.000142 & 64                            & 0.000148 \\
128                           & 0.001713 & 128                           & 0.000036 & 128                           & 0.000037 \\ \hline
Convergence rate                     & -1.3472  & Convergence rate                     & -2.2359  & Convergence rate                     & -2.2610 
\end{tabular}
\caption{Convergence of the Neumann problem on a symmetric triple junction, with right-hand-side $f(x_1,x_2) = \gamma_N( \frac{|x_1|x_1^3}{4})$. Errors computed against an overkill solution with $256$ points per arc. }
\label{tab:p11}
\end{table}

For the second experiment, we consider the same configuration, but with the right-hand side
$$f(x_1,x_2) = \gamma_N (2x_1+x_2).$$
The results are provided in \Cref{tab:p12}. In contrast to the previous case, we are not able to achieve convergence rates greater than one, even when considering different refinement strategies. Based on the discussion in \Cref{sec:NeuProb}, we suspect that the difference between the two cases is that, in the first one, the right-hand side belongs to the range of $\mathcal{W}$, whereas in the second one it does not. 

Indeed, if we further analyze the solution of the second case near the branch point, we observe the presence of a potential jump discontinuity; see \Cref{fig:solNeumann}. The error is concentrated precisely at the branch point, as shown in \Cref{fig:errorNeumann}. Moreover, as is well known, functions exhibiting a jump discontinuity cannot belong to $\widetilde{H}^{\half}(\Gamma)$. This suggest that, in this case, the solution does not lie in the aforementioned space. For comparison, the solution and the errors for the first case are presented in \Cref{fig:main1}.

\begin{table}[ht]
\begin{tabular}{ll|ll|ll}
\multicolumn{2}{l|}{$a = 1.1, b= 1.1$}                & \multicolumn{2}{l|}{$a = 1.1, b= 3.1$}                & \multicolumn{2}{l}{$a = 3.1, b= 3.1$}                \\ \hline
\begin{tabular}[c]{@{}l@{}}Points\\ per arc\end{tabular} & L2- Error & \begin{tabular}[c]{@{}l@{}}Points\\ per arc\end{tabular} & L2- Error & \begin{tabular}[c]{@{}l@{}}Points\\ per arc\end{tabular} & L2- Error \\
8                            & 0.822352 & 8                            & 0.765724 & 8                            & 0.777658 \\
16                            & 0.511187 & 16                            & 0.488669 & 16                            & 0.454613 \\
32                            & 0.323321 & 32                            & 0.314334 & 32                            & 0.288772 \\
64                            & 0.189299 & 64                            & 0.185823 & 64                            & 0.181806 \\
128                           & 0.085397 & 128                           & 0.084296 & 128                           & 0.085262 \\ \hline
Convergence rate                     & -0.8169  & Convergence rate                     & -0.7958  & Convergence rate                     & -0.7973 
\end{tabular}
\caption{Convergence of the Neumann problem on a symmetric triple junction, with right-hand-side $f(x_1,x_2) = \gamma_N( 2x_1+x_2)$. Errors computed against an overkill solution with $256$ points per arc. }
\label{tab:p12}
\end{table}

\begin{figure}
  \centering
  \begin{subfigure}{0.45\textwidth}
    \centering
    \includegraphics[width=\textwidth]{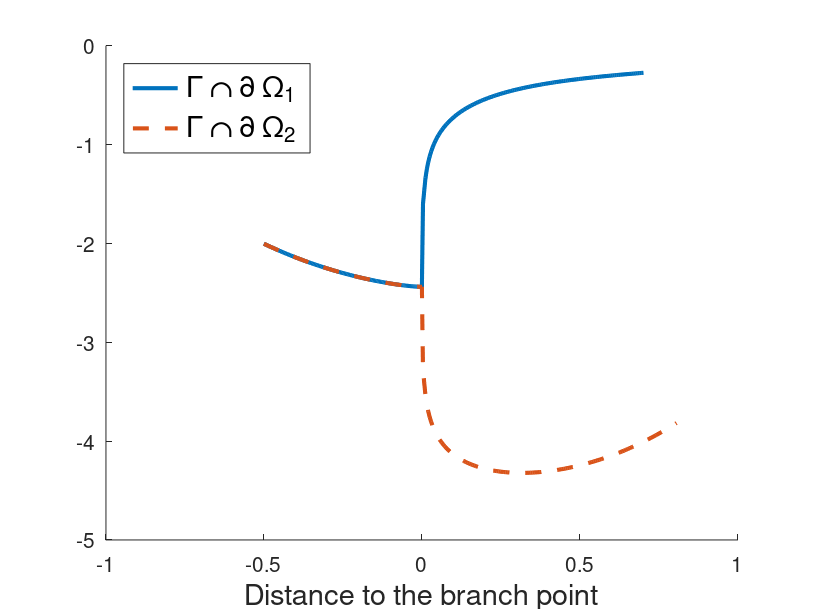}
    \caption{Plot of the solution $\mu$ as a function of the distance to the branch point. }
    \label{fig:solNeumann}
  \end{subfigure}
  \hfill
  \begin{subfigure}{0.42\textwidth}
    \centering
    \includegraphics[width=\textwidth]{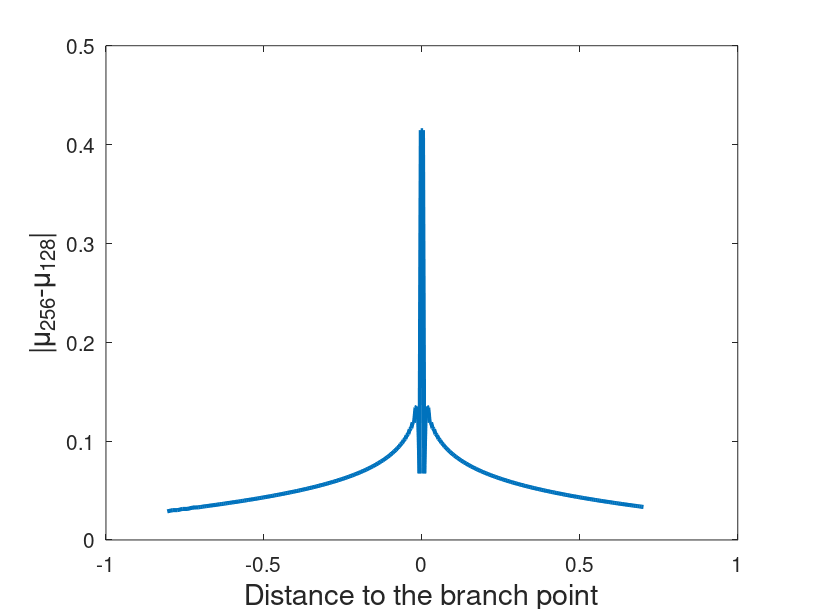}
    \caption{Difference between the solution with 128 and 256 points per arc, as a function of the distance to the branch point.}
    \label{fig:errorNeumann}
  \end{subfigure}
  \caption{Solution and error of the Neumann problem for $f(x_1,x_2) =\gamma_N( 2x_1-x_2)$. The reference variable is the distance to the branch point, with negative sign on $\Gamma_1$, and positive in $\Gamma_2$ and $\Gamma_3$.}
  \label{fig:main}
\end{figure}

\begin{figure}
  \centering
  \begin{subfigure}{0.45\textwidth}
    \centering
    \includegraphics[width=\textwidth]{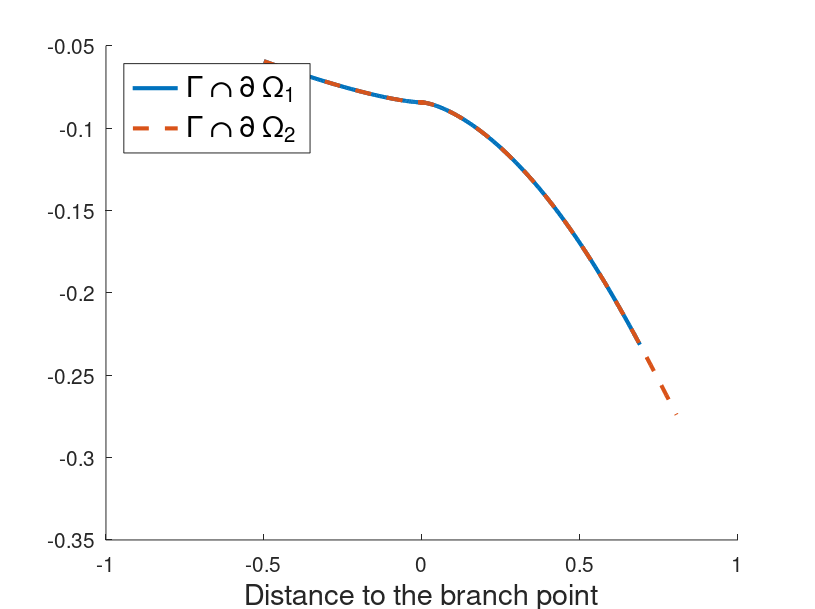}
    \caption{Plot of the solution $\mu$ as a function of the distance to the branch point. }
    \label{fig:solNeumann1}
  \end{subfigure}
  \hfill
  \begin{subfigure}{0.42\textwidth}
    \centering
    \includegraphics[width=\textwidth]{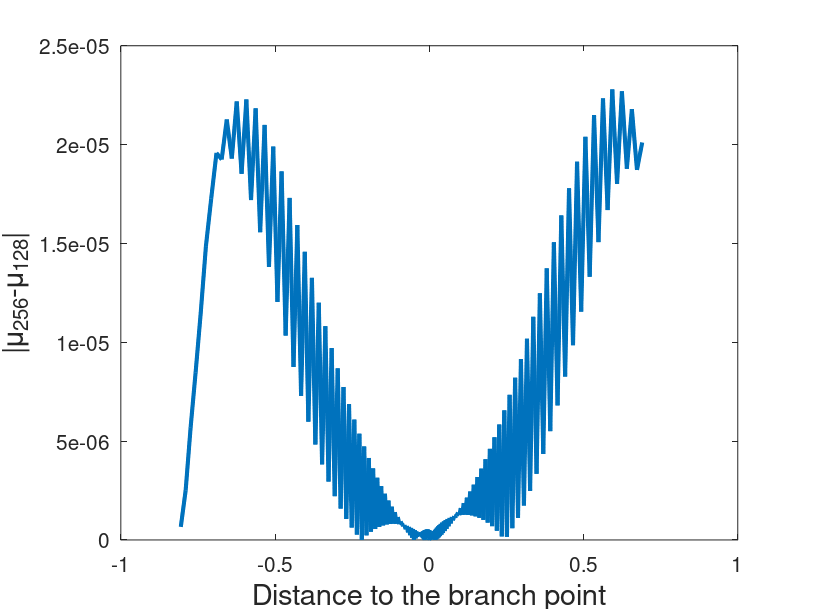}
    \caption{Difference between the solution with 128 and 256 points per arc, as a function of the distance to the branch point.}
    \label{fig:errorNeumann1}
  \end{subfigure}
  \caption{Solution and error of the Neumann problem for $f(x_1,x_2)= \gamma_N( \frac{|x|x^3}{4})$. The reference variable is the distance to the branch point, with negative sign on $\Gamma_1$, and positive in $\Gamma_2$ and $\Gamma_3$.}
  \label{fig:main1}
\end{figure}

  \section{Conclusions and Future Work}
  \label{sec:conclusions}
  
In this work, we extended the classical functional framework for boundary integral equations on open arcs to the more general setting of multi-arcs. This extension allows us to formulate both Dirichlet and Neumann boundary value problems for the Laplace equation as boundary integral equations posed on geometries featuring corners and branch points. For the Dirichlet problem, we established a well-posed integral formulation within a natural fractional Sobolev setting.

A key advantage of the proposed functional framework, compared with the more general approach in \cite{Claeys2013}, is that it yields a simpler and more transparent characterization of the functional spaces associated with multi-arcs. This, in turn, enables the use of classical Galerkin discretization schemes with minimal modification. As a consequence, our analysis provides a rigorous justification for numerical methods that have been previously employed in the literature, including those used in \cite{helsing2024helmholtzdirichletneumannproblems}.
  
From a numerical perspective, the Dirichlet problem appears to be well understood within the present framework. In contrast, the Neumann problem remains more challenging. As demonstrated by the numerical experiments in \Cref{sec:NeumannNumeric}, standard low-order discretizations of the hypersingular operator may fail to achieve satisfactory convergence rates, particularly when the right-hand side does not belong to the range of the operator. This indicates the need for improved discretization strategies or alternative formulations for the Neumann problem on multi-arcs, which we leave as an important direction for future work.
  
Since our integral formulations rely on single- and double-layer potentials defined on open arcs it is natural to expect that several classical analytical results can be generalized to the multi-arc setting. 
  
A central theme of this work has been the study of singular behavior for the Dirichlet problem on multi-arcs. Our numerical experiments confirm that, near the endpoints of a multi-arc, the solution exhibits the same type of singularity as in the classical open-arc case. At branch points, however, the singular behavior is governed by the angles between the intersecting arcs. Based on these observations, we formulate the following conjecture for triple junctions.
  
  \begin{conjecture}
    Consider a triple junction $\Gamma$ formed by the intersection of three line segments $\Gamma_1, \Gamma_2, \Gamma_3$, whose common point is $\bm{0} \in \IR^2$. Let $\theta_{i,j}$ denote the angle between $\Gamma_i$ and $\Gamma_j$ for $i,j \in \{1,2,3\}$, $i \neq j$. If $f$ is the restriction to $\Gamma$ of a smooth function, then the solution $\lambda$ of the integral equation
    $$\mathcal{V}\lambda = f,$$
    exhibits, along $\Gamma_i$, a leading-order singularity of the form
    $$
     t^{\min_{j\neq i}\left( \frac{\pi}{\theta_{i,j}}\right)-1}, \quad i =1,2,3,
    $$
    where $t$ is the distance to the branch point $\bm{0}$. 
  \end{conjecture}
We expect that this conjecture can be extended to more general multi-arc configurations with higher-order junctions. A rigorous analytical proof, together with a systematic study of its implications for numerical discretizations, constitutes an important topic for future research.

\bibliography{biblio}
\bibliographystyle{siam}

\appendix
\section{Proof of \Cref{lemma:dens} for Open Arcs}
\label{app:profDensOpenArc}

Let us consider the case where $\Gamma$ is an open arc with Lipchitz paramatrization. We adapt the argument of \cite[Thm.~3.40]{mclean2000strongly} to this setting. 

First consider $\Gamma = (-1,1)\times \{0\}$, and let $E = \{ U \in H^{1}(\IR^2): \gamma_D U= 0\}$. We take $\ell \in E^\star$ such that $\ell(\Phi) = 0$ for all $\Phi \in \C_{0}^\infty(\IR^2 \setminus \overline{\Gamma})$. Then there exist $W \in H^{-1}(\IR^2)$ such that $\langle W , \Phi\rangle_{H^{-1}(\IR^2), H^1(\IR^2)} = \ell(\Phi) $ for all $\Phi \in E$. This implies, by definition of the support of a distribution, that $W \in H^{-1}_{\overline{\Gamma}} \coloneq \{F \in H^{-1}(\IR^2) : \supp F \subset \overline{\Gamma}\}$. We notice that $H^{-1}_{\overline{\Gamma}} \subset H^{-1}_{\IR \times \{0\}}$, then by \cite[Lemma 3.39]{mclean2000strongly}, we have the following representation 
  $$
  W = v \otimes \delta, 
  $$
  where $v \in H^{-\frac{1}{2}}(\IR)$, and $\delta $ is the standard Dirac distribution. Since $\supp W \subset [-1,1]\times\{0\}$, we have that $\supp v \subset [-1,1]$, and by \cite[Theorem 3.29]{mclean2000strongly} we obtain that $v \in \widetilde{H}^{-\frac{1}{2}}((-1,1))$. Then for $\Psi \in E$ we have that 
  $$
  \ell(\Psi) = \langle W , \Psi\rangle_{H^{-1}(\IR^2),H^1(\IR^2)} = \langle v, \gamma_D \Psi\rangle_{(-1,1)} =0. 
  $$
  which proofs the results for the canonical arc $\Gamma = (-1,1)\times \{0\}$. For other open arcs the result is extended using the parametrization function. 

\section{Computation of Poles}
\label{app:compPoles}

We compute the poles of the functions arising in the Mellin-domain analysis. We distinguish two cases.
\begin{itemize}
  \item Case 1: We consider the poles of the function 
  $$ \frac{(s+i)\sinh(\pi (s+i))}{\cosh(\pi(s+i))-\cosh((\pi-\theta)(s+i))} \widehat{g}_j((s+i)), \quad j=0,1.$$
We decompose the analysis into the following contributions:
  \begin{enumerate}
    \item The zeros of $\sinh(\pi (s+i))$, are given by $$s = ki \quad k \in \IZ.$$ 
    \item 
    The poles of $\widehat{g}_j((s+i))$ are determined using the fact that $g$ is a linear combination of the functions $g_0,g_1,g_2$ which are analytic. Following \cite{heuer1996hp}, we conclude that the poles are located at
    $$s = ik, \quad k =0 \cup \IN.$$ 
    Notice that $s =-i$ is not a pole, since by definition of $g_j$, $j=0,1$ are continuous at the origin. 
    \item 
    To determine the zeros of $$\cosh(\pi(s+i))-\cosh((\pi-\theta)(s+i))$$, we use an hyperbolic function identity, so these are equal to the zeros of
    $$
    \sinh(\frac{\theta s}{2}) \sinh(\frac{2\pi-\theta}{2}s).
    $$
    Then by replacing $\theta = \frac{2 \pi}{3}$, we obtain two families of first order zeros, $s = i(3k-1)$ and $s = i(\frac{3k}{2}-1)$, for $k=0 \cup \IN$. We have ignored zeros such that the imaginary part of $s$ is smaller that $-1$, as these would imply that the overall solution is outside of the standard Sobolev spaces for the Dirichlet problem, as in \cite{heuer1996hp}.
  \end{enumerate}
  Collecting all contributions and recalling that $s=-i$ is not a pole, we obtain that $s = i(3k-1)$, $k=1,2,\hdots$, are second order poles and $s = i\frac{6k-5}{2}$, $k=1,2,\hdots$ are simple poles. 
  \item Case 2: We now consider the function
  $$  \frac{(s+i)\sinh(\pi (s+i))}{\cosh(\pi(s+i))+2\cosh((\pi-\theta)(s+i))} \widehat{g}_2((s+i)), $$
 The analysis of the numerator is identical to Case~1, so we focus on the zeros of the denominator. These are characterized by the equation
  $$
  \cosh(\pi x)+2\cosh(\frac{\pi}{3}x)=0,
  $$
  $$
  e^{\pi x} +e^{-\pi x}+2(e^{\frac{\pi}{3} x}+e^{-\frac{\pi}{3} x})=0.
  $$
  Writing this equation in exponential form yields
    $$
  e^{\frac{\pi}{3} x}(e^{\frac{2\pi}{3} x}+2)+ e^{-\frac{\pi}{3} x}(e^{-\frac{2\pi}{3} x}+2)=0.
  $$
  Introducing the change of variables $z = e^{\frac{\pi}{3}x}$, we obtain
      $$
  z(z^2+2)+\frac{1}{z}\left(z^{-2}+2\right)=0.
  $$
Solving this equation gives
  $$
  z = \pm i , z = \pm \half(1+i\sqrt{3}), z =\pm \half(1-i\sqrt{3}) . 
  $$
Taking logarithms, we obtain the corresponding values of $x$,
  $$
  x = i \left( \threehalf+3k \right), \ x = i(1+3k),\ x = i(2+3k), \quad k \in \IZ$$ 
Consequently, the poles are located at
  $$ s = i \left( \half+3k \right), s = i(3k),\ s = i(1+3k), \quad k \in 0 \cup \IN.$$
\end{itemize}

\end{document}